\documentclass{article}%
\usepackage{amsmath}
\usepackage{amsfonts}
\usepackage{amssymb}
\usepackage{graphicx}%
\usepackage{mathrsfs}
\providecommand{\U}[1]{\protect \rule{.1in}{.1in}}
\newtheorem{theorem}{Theorem}[section]

\newtheorem{corollary}[theorem]{Corollary}

\newtheorem{definition}[theorem]{Definition}
\newtheorem{example}[theorem]{Example}

\newtheorem{lemma}[theorem]{Lemma}

\newtheorem{remark}[theorem]{Remark}

\newenvironment{proof}[1][Proof]{\noindent \textbf{#1.} }{\  \rule{0.5em}{0.5em}}
\begin{document}

\title{Stochastic  optimal control problem with infinite horizon  driven by $G$-Brownian motion}
\author{Mingshang Hu \thanks{Zhongtai Securities Institute for Financial Studies, Shandong University, Jinan 250100, China. humingshang@sdu.edu.cn. Research supported by the National Natural Science Foundation of China (No. 11671231) and the Young Scholars Program of Shandong University (No. 2016WLJH10).}
\and Falei Wang\thanks{Corresponding author. Zhongtai Securities Institute for Financial  Studies and Institute for Advanced Research, Shandong University, Jinan 250100, China.
flwang2011@gmail.com. Research supported by   the National Natural Science Foundation of China (No. 11601282), the Natural Science Foundation of Shandong Province (No. ZR2016AQ10), the Fundamental Research Funds of Shandong University (No. 2015GN023) and  the China
Scholarship Council (No.201606225002).
Hu and Wang's research was
partially supported by the Tian Yuan Projection of the National Natural Sciences Foundation of China (No. 11526205) and the 111 Project (No. B12023).}}
\date{}
\maketitle
\begin{abstract}
The present paper considers a stochastic  optimal control
problem,  in which the cost function is defined through  a backward
stochastic differential equation with infinite horizon driven by $G$-Brownian motion. Then we
study the regularities of the value function and  establish the dynamic programming principle.
Moreover, we prove that the value function is the unique viscosity solution of the related  Hamilton-Jacobi-Bellman-Isaacs (HJBI) equation.
\end{abstract}


\textbf{MSC-classification}: 93E20, 60H10, 35J60
\section{Introduction}
It is well-known that the backward stochastic differential
equations (BSDEs)  theory provides a powerful tool for the study of stochastic recursive
optimal control problem, which generalizes the classical
stochastic optimal control problem.
Indeed,  Peng \cite{peng-dpp} established a generalized dynamic programming principle (DPP) and provided a probabilistic
 interpretation for a wide class of  Hamilton-Jacobi-Bellman (HJB) equations.
 Afterwards, Peng  \cite{peng-1993} and \cite{peng-dpp-1} introduced the  ``backward semigroup'' approach and  extended  the previous results to more general case.
For further research on this topic,  the reader is referred to \cite{BJ, BN, MY-19992, J.Yong}   and the references
therein.

Recently, Peng  introduced a time-consistent fully nonlinear expectation theory.
 As a typical and important case, Peng established the $G$-expectation theory
 (see \cite {P10}). Under the $G$-expectation framework,  the stochastic integral
with respect to $G$-Brownian motion
  was also stated.
  Then Peng \cite{P10} and Gao \cite{G} obtained
 the existence and uniqueness theorem
for stochastic differential equations  driven by $G$-Brownian motion ($G$-SDEs). Moreover,
 Hu et.al. \cite{HJPS,HJPS1} introduced the backward stochastic differential
equations driven by $G$-Brownian motion ($G$-BSDEs).
The $G$-expectation theory provides a useful tool for studying financial problems
under volatility uncertainty. Indeed, with the help of $G$-stochastic analysis theory, Epstein and Ji \cite{EJ-2,EJ-1} studied a
recursive utility problem under both mean and volatility uncertainty, which generalizes the ones of \cite{CE}.
In a different setting, Soner, Touzi and Zhang \cite{STZ1} established the so-called 2BSDEs theory, which shares many similarities with $G$-BSDEs.

Recently, Hu and Ji \cite{HJ1} (see also \cite{HJY}) considered a stochastic recursive optimal control problem under volatility uncertainty.
Since there is no dominated probability measure in the $G$-framework, it is much more complicated than the classical case. In particular,
the essential infimum of a family of
random variables may not exist  and
it is difficult to construct a  discrete approximation of an admissible control  to get the dynamic
 programming
 principle in the nonlinear case. With the help of  quasi-surely stochastic analysis theory (see \cite{DHP11} and \cite{DenisMartini2006}), they introduced an ``implied partition'' approach to
 establish the DPP and got that  the value function
is the viscosity solution to the following HJBI equation:
\begin{equation*}
\left\{
\begin{array}
[c]{l}%
\partial_{t}V+ \underset{u\in U}{\inf}[G(H(x,V,\partial_{x}%
V,\partial_{xx}^{2}V,u))+\langle \partial_{x}%
V,b(x,u)\rangle+f(x,V,\partial_{x}%
V\sigma(x,u),u)]=0,\\
V(T,x)=\phi(x),
\end{array}
\right.
\end{equation*}
 which generalizes the ones of Peng \cite{peng-1993}.

Motivated by \cite{HJ1, peng-dpp},  we shall study the following  HJBI equation:
\begin{equation}\label{myqfl}
\underset{u\in U}{\inf}[G(H(x,V,\partial_{x}%
V,\partial_{xx}^{2}V,u))+\langle \partial_{x}%
V,b(x,u)\rangle+f(x,V,\partial_{x}%
V\sigma(x,u),u)]=0,
\end{equation}
which is a fully nonlinear  elliptic partial differential equation (PDE) in $\mathbb{R}^n$.  We refer the reader to \cite{BCI-2008,BN,L-1983,LM-1982,LM-19822} for a closest related approach, where the related PDEs are
HJB equations with Dirichlet boundary.

This paper is devoted to providing a stochastic representation for the
viscosity solution to the  HJBI equation \eqref{myqfl}. A key ingredient of our approach is  based on the $G$-BSDEs theory with infinite horizon,
which is   introduced by \cite{HW} through combing
nonlinear stochastic analysis method with the linearization approach
formulated by \cite{BH} (see also \cite{FH1, RM}).   Indeed, consider the following
   $G$-FBSDE with infinite horizon:
 \begin{align}\label{yw171}
 \begin{cases}
&X_{s}^{0,x,u}=x+\int^s_0b(X_{r}^{0,x,u},u_r)dr+\int^s_0h_{ij}(X_{r}^{0,x,u},u_r)d\langle
B^i,B^j\rangle_{r}+\int^s_0\sigma(X_{r}^{0,x,u},u_r)dB_{r},\\
&
Y_{s}^{0,x,u}   =Y_{T}^{0,x,u}+\int_{s}^{T}f(X_{r}^{0,x,u}%
,Y_{r}^{0,x,u},Z_{r}^{0,x,u},u_r)dr-\int_{s}^{T}Z_{r}^{0,x,u}dB_{r}\\
& \ \ \ \ \ \ \ +\int_{s}^{T}g_{ij}(X_{r}^{0,x,u}%
,Y_{r}^{0,x,u},Z_{r}^{0,x,u},u_r)d\langle B^i,B^j\rangle_{r} -(K_{T}^{0,x,u}-K_{s}^{0,x,u}).
\end{cases}
\end{align}
 The value function of our stochastic optimal control problem
 is given by
 \[
 V(x):=\underset{u\in\mathcal{U}[0,\infty)}{\inf}Y_{0}^{0,x,u}.
 \]
Since $G$ is a sublinear function, our stochastic control problem is essentially a `` $\inf\sup$ problem'', which can be seen as a robust optimal
control problem. For recent important
developments of this field, we refer
the readers to \cite{DK,MP,NN, TTU}. In \cite{DK},  a duality theory for robust utility maximization is stated in a non-dominated model.
In \cite{MP}, the authors applied 2BSDE with
quadratic growth to study robust utility maximization problem and
\cite{TTU} studied robust exponential and power utilities in a different setting. In \cite{NN}, the authors dealt with  a robust portfolio optimization problem
in a continuous-time financial market with jumps.

A potential application of this paper is to study the problems of minimizing an infinite horizon, discounted expected cost under  volatility uncertainty:
\[
J(x,u)=\mathbb{\hat{E}}[\int^{\infty}_0 \exp(-\lambda s)\psi(X^{0,x,u}_s,u_s)ds],
\]
where $\lambda>0$ is a discount factor and $\psi(x,u)$ is a cost function.
Indeed, taking $f(x,y,u)=-\lambda y+\psi(x,u)$  and $g_{ij}=0$ in the equation \eqref{yw171}, we have
\[
Y_{s}^{0,x,u}   =Y_{T}^{0,x,u}+\int_{s}^{T}(-\lambda Y_{r}^{0,x,u}+\psi(X_{r}^{0,x,u},u_r) )dr-\int_{s}^{T}Z_{r}^{0,x,u}dB_{r}
 -(K_{T}^{0,x,u}-K_{s}^{0,x,u}).
\]
By change of variable formula, we have
\[Y_{0}^{0,x,u}=\mathbb{\hat{E}}[\exp(-\lambda T)Y_{T}^{0,x,u}
+\int_{0}^{T}\exp(-\lambda r)\psi(X_{r}^{0,x,u},u_r) dr].
\]
Note that the expectation of $|Y^{0,x,u}_T|$ is uniformly bounded (see section 3). Then  sending $T\uparrow\infty$ yields that
\[
J(x,u)=Y_{0}^{0,x,u}.
\]
Thus the above stochastic  optimal control theory with infinite horizon   provides an alternative way for studying this problem.
In the linear case, more research on this topic  can be found in  \cite{ Fleming W.H,HY} and the references
therein.

The objective of our paper is  to prove that the  value function $V$  is the viscosity solution of
the HJBI equation \eqref{myqfl}. First, we
investigate the  properties of the value function $V$ by the $G$-stochastic analysis approach, which is different from the ones in \cite{HJ1} since the
cost function equation is a $G$-BSDE with infinite horizon.
Then we obtain the  following relation
\[
 V(x)=\underset{u\in\mathcal{U}[0,\infty)}{\inf}Y_{0}^{0,x,u}=\underset{u\in\mathcal{U}[t,\infty)}{ess \inf}Y_{t}^{t,x,u},
\]
which  is crucial to  give a stochastic representation for  the  HJBI equation \eqref{myqfl}.
Next we establish the DPP by the ``backward semigroup'' method and a new version of ``implied partition'' approach.
This provides a fundamental tool for the study of the stochastic control problems in the $G$-framework.  Finally, we show that the value function is the  viscosity solution of the
HJBI equation \eqref{myqfl} and a stochastic verification theorem is also stated. Moreover, based on stochastic control approach  and the method introduced in \cite{HW}, we also get the uniqueness of viscosity solution to equation \eqref{myqfl}.

The uniqueness of viscosity solutions of elliptic PDEs in $\mathbb{R}^n$ has been studied for various types of HJB equations of second
order (see, e.g. \cite{CMI}, \cite{N1} and \cite{PE}). In \cite{CMI}, a result is stated under some uniformly continuous assumptions for $H$.
In \cite{N1} and \cite{PE}, the authors both dealt with semi-linear elliptic PDEs under  locally uniformly continuous conditions for $H$. However, they both assumed some additional conditions, such as condition (6.13) in \cite{PE} and bounded condition (4.2) on diffusion term in \cite{N1}.
In this paper, we treat the fully nonlinear case  under some locally uniformly continuous conditions for $H$ and remove these additional conditions (see also \cite{HW} for the case
there is no control).
  However, we only consider  viscosity solutions of quadratic  growth.
 On the other hand, Ren \cite{RZ}
studied the viscosity solutions  of fully nonlinear elliptic path-dependent PDEs under some uniformly continuous conditions for $H$ (see \cite{RZ1} for more research on this topic), which provides  an important framework  for the study of non-Markovian stochastic  control problem with infinite horizon.

The paper is organized as follows. In  section 2, we present some preliminaries
for $G$-Brownian motion  and $G$-BSDEs theory.
We state our  stochastic optimal
control problem in section 3. The section 4 is devoted to studying  the regularities of the value function. In
section 5, we  prove that the
value function is the unique  viscosity solution of the related HJBI equation.

\section{Preliminaries}

The main purpose of this section is to recall some basic notions and
results of $G$-expectation and $G$-BSDEs, which are needed in the sequel. The
readers may refer to \cite{P07a}, \cite{P08a} and \cite{P10} for more
details.
\subsection{$G$-Brownian motion}
Let $\Omega=C_{0}^{d}(\mathbb{R}^{+})$ be the space of all $\mathbb{R}^{d}%
$-valued continuous paths $(\omega_{t})_{t\geq0}$ starting from origin,
equipped with the distance%
\[
\rho(\omega^{1},\omega^{2}):=\sum_{i=1}^{\infty}2^{-i}[(\max_{t\in
\lbrack 0,i]}|\omega_{t}^{1}-\omega_{t}^{2}|)\wedge1].
\]
For each $t\in \lbrack0,\infty)$, we denote
\begin{itemize}
\item $B_{t}(\omega):=\omega_{t}$ for each $\omega \in \Omega$;
\item $\mathcal{B}(\Omega)$: the Borel $\sigma$-algebra of $\Omega$,\ $\Omega_{t}:=\{ \omega_{\cdot \wedge t}:\omega \in \Omega \}$,\
$\mathcal{F}_{t}:=\mathcal{B}(\Omega_{t})$;
\item $L^{0}(\Omega)$: the space of all $\mathcal{B}(\Omega)$-measurable real
functions;

\item $L^{0}(\Omega_{t})$: the space of all $\mathcal{B}(\Omega_{t}%
)$-measurable real functions;

\item $C_{b}(\Omega)$: all bounded continuous elements in $L^{0}(\Omega)$;
$C_{b}(\Omega_{t}):=C_{b}(\Omega)\cap L^{0}(\Omega_{t})$;
\item $L_{ip}(\Omega):=\{ \varphi(B_{t_{1}},\ldots,B_{t_{k}}):k\in \mathbb{N}%
,t_{1},\ldots,t_{k}\in \lbrack0,\infty),\varphi \in C_{b.Lip}(\mathbb{R}%
^{k\times d })\}$, where $C_{b.Lip}(\mathbb{R}^{k\times d})$ denotes the space of all
bounded and
Lipschitz functions on $\mathbb{R}^{k\times d}$; $L_{ip}(\Omega_{t}):=L_{ip}%
(\Omega)\cap L^{0}(\Omega_{t})$.
\end{itemize}

Given a monotonic and sublinear function $G:\mathbb{S}%
(d)\rightarrow \mathbb{R}$,
let the canonical process
$B_{t}=(B_{t}^{i})_{i=1}^{d}$ be the $d$-dimensional $G$-Brownian
motion on the $G$-expectation space $(\Omega,L_{ip}(\Omega),\mathbb{\hat{E}%
}[\cdot],(\mathbb{\hat{E}}_{t}[\cdot])_{t\geq0})$, where $\mathbb{S}(d)$ denotes the space of all $d\times d$ symmetric matrices. For each
$p\geq1$, the completion of $L_{ip}(\Omega)$ under the norm
$||X||_{L_{G}^{p}}:=(\mathbb{\hat{E}}[|X|^{p}])^{1/p}$ is denoted by
$L_{G}^{p}(\Omega)$. Similarly, we can define
$L_{G}^{p}(\Omega_{T})$ for each fixed $T\geq0$. In this paper, we always assume
 that $G$ is non-degenerate, i.e., there exist some constants
$0<\underline{\sigma}^{2}\leq\bar{\sigma}^{2}<\infty$ such that
\[
\frac{1}{2}\underline{\sigma}^{2}\mathrm{tr}[A-B]\leq G(A)-G(B)\leq \frac{1}{2}\bar{\sigma}^{2}\mathrm{tr}[A-B]\text{ for }A\geq B.
\]
 Then there exists a bounded and closed subset $\Gamma
\subset\mathbb{S}^+(d)$ such that%
\begin{equation*}
G(A)=\frac{1}{2}\underset{Q\in\Gamma}{\sup}\mathrm{tr}[AQ],
\label{G-representation}%
\end{equation*}
where $\mathbb{S}^+(d)$ denotes the space of all $d\times d$ symmetric positive definite matrices.

\begin{theorem}[\cite{DHP11,HP09}]
\label{the2.7}  There exists a weakly compact set
$\mathcal{P}$ of probability
measures on $(\Omega,\mathcal{B}(\Omega))$, such that
\[
\mathbb{\hat{E}}[\xi]=\sup_{P\in\mathcal{P}}E_{P}[\xi]\  \text{for
 all}\ \xi\in  {L}_{G}^{1}{(\Omega)}.
\]
$\mathcal{P}$ is called a set that represents $\mathbb{\hat{E}}$.
\end{theorem}

Let $\mathcal{P}$ be a weakly compact set that represents
$\mathbb{\hat{E}}$.
For this $\mathcal{P}$, we define capacity%
\[
c(A):=\sup_{P\in\mathcal{P}}P(A),\ A\in\mathcal{B}(\Omega).
\]
A set $A\subset\mathcal{B}(\Omega)$ is polar if $c(A)=0$.  A
property holds $``quasi$-$surely''$ (q.s.) if it holds outside a
polar set. In the following, we do not distinguish between two random
variables $X$ and $Y$ if $X=Y$ q.s.

\begin{definition}[\cite{P08a}]
\label{def2.6} Let $M_{G}^{0}(0,T)$ be the collection of processes
of  the following form: for a given partition
$\{t_{0},\cdot\cdot\cdot,t_{N}\}$ of $[0,T]$,
\[
\eta_{t}(\omega)=\sum_{i=0}^{N-1}\xi_{i}(\omega)\mathbf{1}_{[t_{i},t_{i+1})}(t),
\]
where $\xi_{i}\in L_{ip}(\Omega_{t_{i}})$,
$i=0,1,2,\cdot\cdot\cdot,N-1$. For each $p\geq1$,  denote by
$M_{G}^{p}(0,T)$ the completion of $M_{G}^{0}(0,T)$ under the norm
$||\eta||_{M_{G}^{p}}:=(\mathbb{\hat{E}}[\int_{0}^{T}|\eta_{t}|^{p}dt])^{1/p}$.
\end{definition}

For each $1\leq i, j \leq d$, we  denote by $\langle B^i, B^j\rangle$   the mutual variation process.
Then for two processes $ \eta\in M_{G}^{2}(0,T)$ and $ \xi\in M_{G}^{1}(0,T)$,
the $G$-It\^{o} integrals  $\int^{\cdot}_0\eta_sdB^i_s$ and $\int^{\cdot}_0\xi_sd\langle
B^i,B^j\rangle_s$  are well defined, see  Li-Peng \cite{L-P}  and Peng \cite{P10}.
Moreover, we also have the corresponding $G$-It\^{o} formula.

Consider the following $G$-It\^{o} process (in this paper we always use Einstein convention)
\[
X_{t}^{\nu}=X_{0}^{\nu}+\int_{0}^{t}\alpha_{s}^{\nu}ds+\int_{0}^{t}\eta
_{s}^{\nu ij}d\left \langle B^{i},B^{j}\right \rangle
_{s}+\int_{0}^{t}\beta _{s}^{\nu j}dB_{s}^{j},
\]
where $\nu=1,\ldots,n$.
\begin{theorem}[\cite{L-P,P08a}]\label{Thm6.5}
Suppose that $\Phi$ is a $C^2$-function on $\mathbb{R}^n$ such that
$\partial_{x^{\mu}x^{\nu}}^{2}\Phi$ is a function of polynomial growth
for any $\mu,\nu=1,\cdots,n$. Let $\alpha^{\nu}$, $\beta^{\nu
j}$ and $\eta^{\nu ij}$, $\nu=1,\cdots,n$, $i,j=1,\cdots,d$ be
in $M_{G}^{2}(0,T)$. Then for each $t\geq0$ we have
\begin{align}
\Phi(X_{t})-\Phi(X_{s}) &
=\int_{s}^{t}\partial_{x^{\nu}}\Phi(X_{u})\beta _{u}^{\nu
j}dB_{u}^{j}+\int_{s}^{t}\partial_{x^{\nu}}\Phi(X_{u})\alpha_{u}^{\nu
}du\label{e629}\\
& \ \ \ +\int_{s}^{t}[\partial_{x^{\nu}}\Phi(X_{u})\eta_{u}^{\nu ij}+\frac{1}%
{2}\partial_{x^{\mu}x^{\nu}}^{2}\Phi(X_{u})\beta_{u}^{\mu
i}\beta_{u}^{\nu j}]d\left \langle B^{i},B^{j}\right \rangle
_{u} \ q.s.\nonumber
\end{align}

\end{theorem}

\subsection{$G$-BSDEs}
For a fixed real number $T>0$, consider the following type of $G$-BSDEs:%
\begin{align}\label{e3}
Y_{t}    =&\xi+\int_{t}^{T}f(s,Y_{s},Z_{s})ds+\int_{t}^{T}g_{ij}(s,Y_{s}%
,Z_{s})d\langle B^{i},B^{j}\rangle_{s} -\int_{t}^{T}Z_{s}dB_{s}\nonumber\\
& -(K_{T}-K_{t}),  \ \   q.s. %
\end{align}
where
\[
f(t,\omega,y,z),g_{ij}(t,\omega,y,z):[0,T]\times\Omega\times
\mathbb{R}\times\mathbb{R}^{d}\rightarrow\mathbb{R}%
\]
satisfy the following properties:
\begin{description}
\item[(H1)] There exists a constant $\beta>0$ such that for any $y,z$,
$f(\cdot,\cdot,y,z),g_{ij}(\cdot,\cdot,y,z)\in M_{G}^{2+\beta}(0,T)$;

\item[(H2)] There exists a constant $L_1>0$ such that
\[
|f(t,y,z)-f(t,y^{\prime},z^{\prime})|+\sum_{i,j=1}^{d}|g_{ij}(t,y,z)-g_{ij}%
(t,y^{\prime},z^{\prime})|\leq L_1(|y-y^{\prime}|+|z-z^{\prime}|).
\]
\end{description}

Let $S_{G}^{0}(0,T)=\{h(t,B_{t_{1}\wedge t},\cdot\cdot\cdot,B_{t_{n}\wedge
t}):t_{1},\ldots,t_{n}\in\lbrack0,T],h\in C_{b,Lip}(\mathbb{R}^{n+1})\}$. For
$p\geq1$ and $\eta\in S_{G}^{0}(0,T)$, set $\Vert\eta\Vert_{S_{G}^{p}}=\{
\mathbb{\hat{E}}[\sup_{t\in\lbrack0,T]}|\eta_{t}|^{p}]\}^{\frac{1}{p}}$.
Denote by $S_{G}^{p}(0,T)$ the completion of $S_{G}^{0}(0,T)$ under the norm
$\Vert\cdot\Vert_{S_{G}^{p}}$.
For simplicity, we denote by $\mathfrak{S}_{G}^{2}(0,T)$ the collection
of all stochastic processes $(Y,Z,K)$ such that $Y\in S_{G}^{2}(0,T)$, $Z\in
M_{G}^{2}(0,T;\mathbb{R}^{d})$, $K$ is a decreasing $G$-martingale with
$K_{0}=0$ and $K_{T}\in L_{G}^{2}(\Omega_{T})$. Then the above $G$-BSDE admits a unique  $\mathfrak{S}_{G}^{2}(0,T)$-solution.
\begin{theorem}[\cite{HJPS}]
\label{the4.1}  Assume that $\xi\in L_{G}^{2+\beta}(\Omega_{T})$
and $f$, $g_{ij}$ satisfy \emph{(H1)}-\emph{(H2)} for some $\beta>0$. Then equation
\eqref{e3} has a unique solution $(Y,Z,K)\in \mathfrak{S}_{G}^{2}(0,T)$.
\end{theorem}

\begin{remark} \label{yw1}{\upshape
Note that there exist non-trivial decreasing and continuous $G$-martingales.
Indeed, $\{\int^{t}_0\xi_s^{ij}d\langle
B^i,B^j\rangle_s-2\int^{t}_0G(\xi_s)ds\}_{0\leq t\leq T}$ is a  typical decreasing $G$-martingale for each $\xi_s^{ij}\in M_{G}^{1}(0,T)$.
Then the martingale representation theorem (MRP) in the $G$-framework is much more complicated than the classical case, see \cite{P10, STZ, Song11}.
}
\end{remark}

Moreover, we have the following estimates.
\begin{theorem}[\cite{HJPS}]
\label{pro3.5}  Let $\xi^{l}\in L_{G}^{2+\beta}(\Omega_{T})$,
$l=1,2$ and $f^{l}$, $g_{ij}^{l}$ satisfy \emph{(H1)}-\emph{(H2)} for some $\beta>0$.
Assume that $(Y^{l},Z^{l},K^{l})\in\mathfrak{S}_{G}^{2}(0,T)$, $l=1,2$ is the solution of equation \eqref{e3} corresponding to the data
$(\xi^{l},f^{l},g_{ij}^{l})$. Set $\hat{Y}_{t}=Y_{t}^{1}-Y_{t}^{2}$.
Then there exists a constant $C$ depending on
$T$, $G$, $L_1$ such that%
\begin{align*}
&|\hat{Y}_{t}|^{2}\leq C \mathbb{\hat{E}}_t[|\hat{\xi}|^{2}+(\int_{0}%
^{T}\hat{h}_{s}ds)^{2}],\\
&\mathbb{\hat{E}}[\sup_{t\in\lbrack0,T]}|\hat{Y}_{t}|^{2}]    \leq
C\{ \mathbb{\hat{E}}[\sup_{t\in\lbrack0,T]}%
\mathbb{\hat{E}}_{t}[|\hat{\xi}|^{2}]]  +\mathbb{\hat{E}}[\sup_{t\in\lbrack0,T]}\mathbb{\hat{E}}_{t}[(\int_{0}%
^{T}\hat{h}_{s}ds)^{2}]]\}.
\end{align*}
where $\hat{\xi}=\xi^{1}-\xi^{2}$ and $\hat{h}_{s}=|f^{1}(s,Y_{s}^{2},Z_{s}%
^{2})-f^{2}(s,Y_{s}^{2},Z_{s}^{2})|+\sum_{i,j=1}^{d}|g_{ij}^{1}(s,Y_{s}%
^{2},Z_{s}^{2})-g_{ij}^{2}(s,Y_{s}^{2},Z_{s}^{2})|$.
\end{theorem}

However, unlike the classical case, the  explicit solutions of linear $G$-BSDEs can only be stated in an auxiliary extended sublinear expectation space.
Suppose that $f({s},Y_s,Z_s)=a_{s}Y_{s}+b_{s}Z_{s}+m_{s}$ and $g_{ij}({s},Y_s,Z_s)=c_{s}^{ij}Y_{s}+d_{s}^{ij}Z_{s}+n^{ij}_{s}$, where $(a_{s})_{s\in\lbrack0,T]}$,
$(c^{ij}_{s})_{ s\in\lbrack0,T]}\in M_{G}^{2}(0,T)$, $(b_{s})_{s\in\lbrack0,T]}$,$(d^{ij}_{s})_{s\in\lbrack0,T]}\in M_{G}^{2}(0,T;\mathbb{R}^d)$ are bounded
processes and $\xi\in L_{G}^{2+\beta}(\Omega_{T})$ for some $\beta>0$,
$(m_{s})_{s\in\lbrack0,T]}$, $(n^{ij}_{s})_{s\in\lbrack0,T]}\in M_{G}^{2}(0,T)$.
Then we construct an auxiliary extended
$\tilde{G}$-expectation space $(\tilde{\Omega},L_{\tilde{G}}^{1}%
(\tilde{\Omega}),\mathbb{\hat{E}}^{\tilde{G}})$ with $\tilde{\Omega}=C_{0}([0,\infty),\mathbb{R}^{2d})$ and%
\begin{align}\label{yw7}
\tilde{G}(A)=\frac{1}{2}\sup_{Q\in\Gamma}\mathrm{tr}\left[  A\left[
\begin{array}
[c]{cc}%
Q & I_d\\
I_d & Q^{-1}%
\end{array}
\right]  \right]  ,\ A\in\mathbb{S}({2d}).
\end{align}
Let $(B_{t},\tilde{B}_{t})_{t\geq 0}$ be the canonical process in the extended space.

\begin{lemma}[\cite{HJPS1}]
\label{the5.2} In the extended $\tilde{G}$-expectation space, the solution of
the linear $G$-BSDE \eqref{e3} can be represented as%
\begin{equation*}
Y_{t}=\mathbb{\hat{E}}_{t}^{\tilde{G}}[\tilde{\Gamma}_{T}^t\xi+\int_{t}^{T}%
m_{s}\tilde{\Gamma}_s^tds+\int_{t}^{T}n_{s}^{ij}\tilde{\Gamma}^t_sd\langle B^i,B^j\rangle_{s}],
\end{equation*}
where $\{\tilde{\Gamma}^t_s\}_{s\in\lbrack t,T]}$ is the solution of the  following $\tilde{G}$-SDE:\begin{equation}
\tilde{\Gamma}^t_s=1+\int_{t}^{s}a_{r}\tilde{\Gamma}^t_rdr+\int_{t}^{s}c^{ij}_{r}\tilde{\Gamma}^t_rd\langle B^i, B^j\rangle
_{r}+\int_{t}^{s}d_{r}^{ij}\tilde{\Gamma}^t_rdB_{r}+\int_{t}^{s}b_{r}\tilde{\Gamma}^t_{r}d\tilde{B}_{r}.
\label{LSDE2}%
\end{equation}
Moreover,
\begin{equation}\label{yw2}
\mathbb{\hat{E}}_{t}^{\tilde{G}}[\tilde{\Gamma}^t_TK_{T}-\int_{t}^{T}a_{s}K_{s}\tilde{\Gamma}^t_s%
ds-\int_{t}^{T}c^{ij}_{s}K_{s}\tilde{\Gamma}^t_sd\langle
B^i, B^j\rangle_{s}]=K_{t}.
\end{equation}
\end{lemma}

\section{Formulation of the problem}
We now introduce the definition of admissible control.  Assume  $U$ is a given compact subset of
$\mathbb{R}^{m}$.
\begin{definition}
For each given $t\geq 0$, $u:[t,\infty)\times\Omega\rightarrow U$ is said to be an admissible control on
$[t,\infty)$, if $u\in M_{G}^{2}(t,\infty;\mathbb{R}^{m})$, where $ M_{G}^{2}(t,\infty;\mathbb{R}^{m})= \underset{T>t}{\cap}M_{G}^{2}(t,T;\mathbb{R}^{m})$, i.e.,
 $\{u_s\}_{0\leq s\leq T}\in M_{G}^{2}(t,T;\mathbb{R}^{m})$ for each $T\geq t$.
 The set of admissible controls on $[t,\infty)$ is denoted by $\mathcal{U}[t,\infty)$. Similarly, we can define $\mathcal{U}[t,T]$.
\end{definition}

For each $t\geq 0$, $u\in\mathcal{U}[t,\infty)$ and $\xi\in L^{p}_G(\Omega_t)$  with $p>2$, consider the
following $G$-SDEs: \begin{align} \label{App1y}
&X_{s}^{t,\xi,u}=\xi+\int^s_tb(X_{r}^{t,\xi,u},u_r)dr+\int^s_th_{ij}(X_{r}^{t,\xi,u},u_r)d\langle
B^i,B^j\rangle_{r}+\int^s_t\sigma(X_{r}^{t,\xi,u},u_r)dB_{r}\end{align}
and $G$-BSDEs with infinite horizon:
\begin{align} \label{App1m}
\begin{cases}&
Y_{s}^{t,\xi,u}   =Y_{T}^{t,\xi,u}+\int_{s}^{T}f(X_{r}^{t,\xi,u}%
,Y_{r}^{t,\xi,u},Z_{r}^{t,\xi,u},u_r)dr-\int_{s}^{T}Z_{r}^{t,\xi,u}dB_{r}\\
& \ \ \ \ \ \ \ +\int_{s}^{T}g_{ij}(X_{r}^{t,\xi,u}%
,Y_{r}^{t,\xi,u},Z_{r}^{t,\xi,u},u_r)d\langle B^i,B^j\rangle_{r} -(K_{T}^{t,\xi,u}-K_{s}^{t,\xi,u}),
\end{cases}
\end{align}
where $b$, $h_{ij}:\mathbb{R}^{n}\times U\rightarrow
\mathbb{R}^{n}$, $\sigma:\mathbb{R}^{n}\times U\rightarrow
\mathbb{R}^{n\times d}$, $f$, $g_{ij}:$
$\mathbb{R}^{n}\times\mathbb{R}\times\mathbb{R}^{d}\times U\rightarrow
\mathbb{R}$ are deterministic continuous functions. For convenience,  set $(X^{x,u},Y^{x,u},Z^{x,u},K^{x,u})=(X^{0,x,u},Y^{0,x,u},Z^{0,x,u},K^{0,x,u})$ for each $(x,u)\in\mathbb{R}^n\times\mathcal{U}[0,\infty).$

In this paper, we shall use the following assumptions:
\begin{description}
\item[(B1)] $h_{ij}=h_{ji}$ and $g_{ij}=g_{ji}$ for each $1\leq i,j\leq d$;
\item[(B2)] There exist some positive constants $L,\alpha_1$ and $\alpha_2$ such that%
\begin{align*}
&|b(x,u)-b(x^{\prime},u^{\prime})|+\sum\limits_{i,j}|h_{ij}(x,u)-h_{ij}(x^{\prime
},u^{\prime})|\leq
L(|x-x^{\prime}|+|u-u^{\prime}|), \\& |\sigma(x,u)-\sigma(x^{\prime},u^{\prime})|\leq
\alpha_1|x-x^{\prime}|+L|u-u^{\prime}|,\\
&  |f(x,y,z,u)-f(x^{\prime},y^{\prime},z^{\prime},u^{\prime})|+
\sum\limits_{i,j}|g_{ij}(x,y,z,u)-g_{ij}(x^{\prime},y^{\prime},z^{\prime},u^{\prime})|\\
&  \leq L((1+|x|+|x^{\prime}|)|x-x^{\prime}|+|y-y^{\prime
}|+|u-u^{\prime}|)+\alpha_2|z-z^{\prime}|;
\end{align*}
\item[(B3)] There exists a constant $\mu>0$  such that $(f(x,y,z,u)-f(x,y^{\prime},z,u))(y-y^{\prime})+2G((g_{ij}(x,y,z,u)-g_{ij}(x,y^{\prime},z,u))(y-y^{\prime}))\leq
-\mu|y-y^{\prime}|^2 $;
\item[(B4)]
$G(\sum\limits_{i=1}^n(\sigma_i(x,u)-\sigma_i(x^{\prime},u))^{\top}(\sigma_i(x,u)-\sigma_i(x^{\prime},u))+2(\langle x-x^{\prime},h_{ij}(x,u)-h_{ij}(x^{\prime},u)\rangle)_{i,j=1}^d )
+\langle x-x^{\prime},b(x,u)-b(x^{\prime},u)\rangle\leq -\eta|x-x^{\prime}|^2$ for some constant $\eta> 0$, where $\sigma_i$ is the $i$-th row of $\sigma$;
\item[(B5)] $\bar{\eta}:=\eta-(1+\bar{\sigma}^2)\alpha_1\alpha_2>0$.
\end{description}

The following estimates about $G$-SDEs can be found in Chapter V of
Peng \cite{P10}.
\begin{lemma}
\label{myq.1}
Under assumption \emph{(B2)},
the $G$-SDE \eqref{App1y} has a unique solution $X^{t,\xi, u}\in M^2_G(t,T)$ for each $T>t$.
Moreover, if $\xi$, $\xi^{\prime}\in L_{G}^{p}(\Omega_{t})$ with $p> 2$, then we have, for each $\delta\in\lbrack0,T-t]$,%
\begin{description}
\item[(i)] $\mathbb{\hat{E}}_{t}[|X_{t+\delta}^{t,\xi,u}-X_{t+\delta}^{t,\xi^{\prime}%
,u^{\prime}}|^{p}]\leq{C}_T(|\xi-\xi^{\prime}|^{p}+\mathbb{\hat{E}}_{t}[\int%
_{t}^{t+\delta}|u_{s}-u^{\prime}_{s}|^{p}ds]);$
\item[(ii)] $\mathbb{\hat{E}}_{t}[\sup\limits_{s\in\lbrack t,T]}|X_{s}^{t,\xi,u}|^{p}]\leq C_{T}(1+|\xi|^{p});$
\item[(iii)] $\mathbb{\hat{E}}_{t}[\sup\limits_{s\in\lbrack t,t+\delta]}|X_{s}^{t,\xi,u}-\xi
|^{p}]\leq C_{T}(1+|\xi|^{p})\delta^{p/2}$,
\end{description}
where the constant $C_{T}$ depends on $L$, $\alpha_1$, $G$, $p$, $n$, $U$ and $T$.
\end{lemma}

We have the following  existence and uniqueness theorem of $G$-BSDE \eqref{App1m} with infinite horizon.
\begin{theorem}\label{HM3}
Let assumptions  \emph{(B1)-(B5)} hold.
Then the $G$-BSDE \eqref{App1m} has a unique solution $(Y^{t,\xi,u}, Z^{t,\xi,u}, K^{t,\xi,u})\in\mathfrak{S}_{G}^{2}(0,\infty)$  such that for some constant $C>0$, \[
|Y^{t,\xi,u}_s|\leq C(1+|X^{t,\xi,u}_s|^2), \ \forall s\geq t \ \ q.s.,
\]
 where
$\mathfrak{S}_{G}^{2}(0,\infty)=\underset{T>0}{\cap}\mathfrak{S}_{G}^{2}(0,T)$.
\end{theorem}
\begin{proof}
The proof will be given in the appendix.
\end{proof}

The aim of our stochastic optimal control problem
is to find some $u \in \mathcal{U}[0, \infty)$ so as to minimise the
objective function $Y_{0}^{x,u}$ for each $x\in\mathbb{R}^{n}$. For this purpose, we define the following value function:
\begin{equation}
V(x):=\underset{u\in\mathcal{U}[0,\infty)}{\inf}Y_{0}^{x,u}\text{ for any }%
x\in\mathbb{R}^{n}. \label{valuefunction01}%
\end{equation}

In order to study the stochastic control problem, we need to define the
essential infimum of $\{Y_{t}^{t,\xi,u}\mid u\in\mathcal{U}[t,\infty)\}.$
\begin{definition}
\label{esssup}For each $\xi\in L^{p}_G(\Omega_t)$ with $p>2$, the essential infimum of $\{Y_{t}^{t,\xi,u}\mid u\in
\mathcal{U}[t,\infty)\}$, denoted by $\underset{u\in\mathcal{U}%
[t,\infty)}{\text{ess}\inf}Y_{t}^{t,\xi,u}$, is a random variable $\zeta\in L_{G}%
^{2}(\Omega_{t})$ satisfying:
\begin{description}
\item[(i)] $\forall u\in\mathcal{U}[t,\infty),$ $\zeta\leq Y_{t}^{t,\xi,u}$ $\ $q.s.$;$

\item[(ii)] if $\eta$ is a random variable satisfying $\eta\leq Y_{t}^{t,\xi,u}$
$\ $q.s. for any $u\in\mathcal{U}[t,\infty)$, then $\zeta\geq\eta$ $\ $q.s.\end{description}
\end{definition}

Then for each $x\in\mathbb{R}^{n}$, we define the following function:
\begin{equation}
V(t,x):=\underset{u\in\mathcal{U}[t,\infty)}{ess\inf}Y_{t}^{t,x,u}\ \text{ for each}\
(t,x)\in[0,\infty)\times\mathbb{R}^{n}. \label{valuefunction0}%
\end{equation}
It is obvious that $V(x)=V(0,x)$.
\begin{remark}{\upshape
At this stage, we cannot even conclude that  $V(t,x)$ exists (see Example 11 in \cite{HJ1}), which is different from the linear  case.
}
\end{remark}
\section{Regularity of the value function}
In this section, we shall study the regularities of the  value function $V$.  In particular, we will prove that $V(t,x)$ is  a deterministic continuous function
independent of the time variable $t$. From now on, if not specified, we always assume (B1)-(B5) hold.

Now recall some notations, which are essentially from \cite{HJ1}:%
\begin{itemize}
\item $L_{ip}(\Omega_{s}^{t})    :=\{ \varphi(B_{t_{1}}-B_{t},...,B_{t_{n}}%
-B_{t}):n\geq1,t_{1},...,t_{n}\in\lbrack t,s],\varphi\in C_{b.Lip}%
(\mathbb{R}^{d\times n})\}$;
\item $L_{G}^{2}(\Omega_{s}^{t})    :=\{ \text{the completion of }L_{ip}(\Omega
_{s}^{t})\text{ under the norm }\Vert\cdot\Vert_{L^2_G}\}$;
\item $M_{G}^{0,t}(t,T)   :=\{ \eta_{s}=\sum_{i=0}^{N-1}\xi_{i}\mathbf{1}_{[t_{i},t_{i+1}%
)}(s):t=t_{0}<\cdots<t_{N}=T,\xi_{i}\in L_{ip}(\Omega_{t_{i}}^{t})\}$;
\item $M_{G}^{2,t}(t,T)   :=\{ \text{the completion of }M_{G}^{0,t}(t,T)\text{
under the norm }\Vert\cdot\Vert_{M_{G}^{2}}\}$;
\item $\mathcal{U}^{t}[t,T]    :=\{u:u\in M_{G}^{2,t}(t,T;\mathbb{R}^{m})\text{ taking
values in }U\}$;
\item $\mathbb{U}[t,T]   :=\{u=\sum\limits_{i=1}^{n}\mathbf{1}_{A_{i}}u^{i}:n\in
\mathbb{N},u^{i}\in\mathcal{U}^{t}[t,T],\mathbf{1}_{A_{i}}\in L_{G}^{2}(\Omega
_{t}),\Omega=%
{\displaystyle\bigcup\limits_{i=1}^{n}}
A_{i}\};$
\item $\mathcal{U}^{t}[t,\infty):=\underset{T>t}{\cap} \mathcal{U}^{t}[t,T]$, $\mathbb{U}[t,\infty):=\underset{T>t}{\cap} \mathbb{U}[t,T]$.
\end{itemize}
\begin{remark}\label{myq6}{
\upshape
Since $U$ is bounded, it is easy to check that  $\eta\in\mathcal{U}[t,\infty)$ belongs to the space $M^p_G(t,T)$ for each $T>t$ and $p\geq 2$.
}
\end{remark}

In order to state the main results of this section,  we shall give some useful estimates in the sequel.
For this purpose, we need to construct an auxiliary extended
$\tilde{G}$-expectation space $(\tilde{\Omega},L_{\tilde{G}}^{1}%
(\tilde{\Omega}),\mathbb{\hat{E}}^{\tilde{G}})$ with $\tilde{\Omega}=C_{0}([0,\infty),\mathbb{R}^{2d})$, where $\tilde{G}$ is given by equation \eqref{yw7}.
Let $(B_{t},\tilde{B}_{t})_{t\geq 0}$ be the corresponding canonical process.

\begin{lemma}\label{HW212}For some given $t\geq 0$, suppose $\Gamma^t$  is the solution of the following $\tilde{G}$-SDE :%
\begin{equation*}
\Gamma^t_{s}=1+\int_{t}^{s}\beta^{1,i}_{r}\Gamma^t_{r}dB^i_{r}+\int_{t}^{s}\beta^{2,i}_{r}\Gamma^t_{r}d\tilde{B}^i_{r}, \ \ s\geq t,
\end{equation*}
 where  $(\beta^{1,i}_{s})_{s\in\lbrack0,\infty)}$,  $(\beta^{2,i}_{s})_{s\in\lbrack0,\infty)} \in M_{G}^{2}(0,\infty)$  are  bounded by $\alpha_2$. Then there is a constant $C_G$ depending only on $G$, such that for each $p\geq 1$,
\[
\mathbb{\hat{E}}^{\tilde{G}}_t[|\Gamma^t_{s}|^p]\leq \exp(C_G(p^2-p)\alpha^2_2(s-t)), \ \ \forall s\geq t\geq 0.
\]
\end{lemma}
\begin{proof}
To simplify presentation, we shall prove only the case that $d=1$, as
other cases can be proved in the same way.
It follows from Proposition 1.3 of Chap. IV  in \cite{P10} that
\[
\Gamma^t_{s}=\exp(\int_{t}^{s}\beta^1_{r}dB_{r}-\frac{1}{2}\int%
_{t}^{s}|\beta^1_{r}|^2d\langle B\rangle_{r}+\int_{t}^{s}\beta^2_{r}d\tilde{B}_{r}-\frac{1}{2}\int_{t}^{s}|\beta^2_{r}|^{2}%
d\langle\tilde{B}\rangle_{r}-\int_{t}^{s}\beta^1_{r}\beta^2_{r}dr).
\]
Thus we conclude that $\hat{\Gamma}^t_s=:\exp(\int_{t}^{s}p\beta^1_{r}dB_{r}-\frac{1}{2}\int%
_{t}^{s}|p\beta^1_{r}|^{2}d\langle B\rangle_{r}+\int_{t}^{s}p\beta^2_{r}d\tilde{B}_{r}-\frac{1}{2}\int_{t}^{s}|p\beta^2_{r}|^{2}%
d\langle\tilde{B}\rangle_{r}-\int_{t}^{s}p^2\beta^1_{r}\beta^2_{r}dr)$ is a $G$-martingale. Then we obtain that
\begin{align*}
\mathbb{\hat{E}}^{\tilde{G}}_t[|\Gamma^t_{s}|^p]\leq\exp(C_G(p^2-p)\alpha^2_2(s-t))\mathbb{\hat{E}}^{\tilde{G}}_t[\hat{\Gamma}_s^t]=\exp(C_G(p^2-p)\alpha^2_2(s-t)),
\end{align*}
where $C_G:=1+\frac{1}{2}(\bar{\sigma}^2+\frac{1}{\underline{\sigma}^2}).$ The proof is complete.
\end{proof}
\begin{lemma}\label{HW2}
Let $\xi,\xi^{\prime}\in L_{G}^{p}(\Omega
_{t};\mathbb{R}^{n})$ with $p>2$ and $u,u^{\prime}\in\mathcal{U}[t,\infty)$. Then there exists a constant ${C}_{\eta}$ depending on $G,\alpha_1,\alpha_2, L,U$ and $\eta$, such that for each $s\geq t$ q.s.
\begin{description}
\item[(i)] $\mathbb{\hat{E}}^{\tilde{G}}_t[|X^{t,\xi,u}_s|^2\Gamma^t_s]\leq C_{\eta}(1+|\xi|^2);$
\item[(ii)] $\mathbb{\hat{E}}^{\tilde{G}}_t[|X^{t,\xi,u}_s-X^{t,\xi^{\prime},u}_s|^2\Gamma^t_s]\leq\exp(-2\bar{\eta} (s-t))|\xi-\xi^{\prime}|^2;$
\item[(iii)]
$
|X^{t,\xi,u}_s-X^{t,\xi,u^{\prime}}_s|^2\Gamma^t_s\leq \exp(\bar{\eta}(t-s))M_s+{C}_{\eta}\int^s_t\exp(\bar{\eta}(r-s))|u_r-u^{\prime}_r|^2\Gamma^t_rdr,
$
where $M$ is a symmetric $\tilde{G}$-martingale. In particular,
\[\mathbb{\hat{E}}^{\tilde{G}}_t[|X^{t,\xi,u}_s-X^{t,\xi,u^{\prime}}_s|^2\Gamma^t_s]\leq {C}_{\eta}\mathbb{\hat{E}}^{\tilde{G}}_t[\int^s_t\exp(\bar{\eta}(r-s))| u_r-u^{\prime}_r|^2 \Gamma^t_rdr], \ \forall s>t.
\]
\end{description}
\end{lemma}
\begin{proof}Without loss of generality, assume $d=1$.
By a similar analysis as in the proof of Lemma 4.1 in \cite{HW}, it is easy to check that (i) and (ii) hold.
Next we shall prove the property (iii). For convenience, we omit superscripts $t$ and $\xi$.

Set
$C_{s}:=\exp(\bar{\eta}(s-t))$.
Applying the $G$-It\^{o} formula \ref{Thm6.5} yields that
\begin{align*}
&C_s|X^{u}_s-X^{u^{\prime}}_s|^2\Gamma^t_s
\\&=\bar{\eta}\int^s_tC_r|X^{u}_r-X^{u^{\prime}}_r|^2\Gamma^t_rdr+2\int^s_tC_r \langle X^{u}_r-X^{u^{\prime}}_r,\bar{b}_r\rangle\Gamma^t_rdr
+\int^s_tC_r\xi_r d\langle B
\rangle_r \\&\ \ \ \  +M_s
+2\int^s_tC_r\langle X^{u}_r-X^{u^{\prime}}_r,\bar{\sigma}_r\rangle \beta^1_r\Gamma^t_rd\langle B\rangle_r +2\int^s_tC_r\langle X^{u}_r-X^{u^{\prime}}_r,\bar{\sigma}_r\rangle \beta^2_r\Gamma^t_rdr,
\end{align*}
where $\bar{\varphi}_s=\varphi(X^{u}_s,u_s)-\varphi(X^{u^{\prime}}_s,u^{\prime}_s)$ for $\varphi=b,h,\sigma$,
$\xi_s=[2\langle X^{u}_s-X^{u^{\prime}}_s,\bar{h}_s\rangle+|\bar{\sigma}_s|^2]\Gamma^t_s$
and \[M_s=2\int^s_tC_r(2\langle X^{u}_r-X^{u^{\prime}}_r,\bar{\sigma}_r\rangle+|X^{u}_r-X^{u^{\prime}}_r|^2\beta^1_r)\Gamma^t_rd B_r
 +\int^s_tC_r|X^{u}_r-X^{u^{\prime}}_r|^2\beta^2_r\Gamma^t_rd\tilde{B}_r.\]

Denote
$\bar{\varphi}^{\prime}_s=\varphi(X^{u}_s,u_s)-\varphi(X^{u^{\prime}}_s,u_s)$ for $\varphi=b,h,\sigma$ and
$\xi^{\prime}_s=2[\langle X^{u}_s-X^{u^{\prime}}_s,\bar{h}^{\prime}_s\rangle+|\bar{\sigma}^{\prime}_s|^2]\Gamma^t_s$.
Note that $\int_{t}^{s}\xi_{r}d\langle B\rangle_{r}-2\int_{t}^{s}G(\xi
_{r})dr\leq0$ from Remark \ref{yw1}. Then we have
\begin{align*}
C_s|X^{u}_s-X^{u^{\prime}}_s|^2\Gamma^t_s\leq \bar{\eta}\int^s_tC_r|X^{x}_r-X^{u^{\prime}}_r|^2\Gamma^t_rdr+M_s+\Pi^1_s+\Pi^2_s,
\end{align*}
where \[\Pi^1_s=2\int^s_tC_r \langle X^{u}_r-X^{u^{\prime}}_r,\bar{b}^{\prime}_r\rangle\Gamma^t_rdr
+2\int^s_tC_rG(\xi^{\prime}_r) dr\]
and \begin{align*} \Pi^2_s=&2\int^s_tC_r \langle X^{u}_r-X^{u^{\prime}}_r,\bar{b}_r-\bar{b}^{\prime}_r\rangle\Gamma^t_rdr
+2\int^s_tC_rG(\xi_r-\xi^{\prime}_r) dr \\&
+2\int^s_tC_r\langle X^{u}_r-X^{u^{\prime}}_r,\bar{\sigma}^{\prime}_r\rangle \beta^1_r\Gamma^t_rd\langle B\rangle_r +
2\int^s_tC_r\langle X^{u}_r-X^{u^{\prime}}_r,\bar{\sigma}^{\prime}_r\rangle \beta^2_r\Gamma^t_rdr\\&
+2\int^s_tC_r\langle X^{u}_r-X^{u^{\prime}}_r,\bar{\sigma}_r-\bar{\sigma}^{\prime}_r\rangle \beta^1_r\Gamma^t_rd\langle B\rangle_r +2\int^s_tC_r\langle X^{u}_r-X^{u^{\prime}}_r,\bar{\sigma}_r-\bar{\sigma}^{\prime}_r\rangle \beta^2_r\Gamma^t_rdr.\end{align*}

Recalling assumption (B4), we obtain that $\Pi^1_s \leq -2\eta\int^s_tC_r|X^{u}_r-X^{u^{\prime}}_r|^2\Gamma^t_rdr.$ From assumption (B2), we have
\begin{align*} &\Pi^2_s
\leq 2(1+\bar{\sigma}^2+\alpha_1\bar{\sigma}^2)L\int^s_tC_r|X^{u}_r-X^{u^{\prime}}_r||u_r-u^{\prime}_r|\Gamma^t_rdr+2\bar{\sigma}^2L^2\int^s_tC_r|u_r-u^{\prime}_r|^2\Gamma^t_rdr
\\ &
+2(1+\bar{\sigma}^2)\alpha_1\alpha_2\int^s_tC_r|X^{u}_r-X^{u^{\prime}}_r|^2\Gamma^t_rdr
+2(1+\bar{\sigma}^2)\alpha_2L\int^s_tC_r|X^{u}_r-X^{u^{\prime}}_r||u_r-u^{\prime}_r|\Gamma^t_rdr.
\end{align*}
Note that
\[
2(1+\bar{\sigma}^2)(1+\alpha_1+\alpha_2)L|X^{u}_s-X^{u^{\prime}}_s||u_s-u^{\prime}_s|\leq \bar{\eta}|X^{u}_s-X^{u^{\prime}}_s|^2+C_{1}|u_s-u^{\prime}_s|^2,
\]
where $C_{1}:=(1+\bar{\sigma}^2)^2(1+\alpha_1+\alpha_2)^2L^2/\bar{\eta}$.

Then by the definition of $\bar{\eta}$, we conclude that
\begin{align}\label{myq3}
&|X^{u}_s-X^{u^{\prime}}_s|^2\Gamma^t_s\leq (C_s)^{-1}M_s+(C_{1}+2\bar{\sigma}^2L^2)(C_s)^{-1}\int^s_tC_r|u_r-u^{\prime}_r|^2\Gamma^t_rdr.
\end{align}
On the other hand, applying H\"{o}lder's inequality, Lemmas \ref{myq.1} (i) and \ref{HW212} yields that
 $M_s$ is a symmetric $\tilde{G}$-martingale, i.e., $M_s$ and $-M_s$ are both $G$-martingales.
Consequently, taking expectation on both sides of equation \eqref{myq3}, we deduce that
\[
\mathbb{\hat{E}}^{\tilde{G}}_t[|X^{u}_s-X^{u^{\prime}}_s|^2\Gamma^t_s]\leq(C_{1}+2\bar{\sigma}^2L^2)\mathbb{\hat{E}}^{\tilde{G}}_t[\int^s_t\exp(\bar{\eta}(r-s))|u_r-u^{\prime}_r|^2\Gamma^t_rdr],
\]
which completes the proof.
\end{proof}

 Note that the constant ${C}_{\eta}$ is independent of $s$, which is
crucial	for our main results. We remark that the above results can be extended to more general case. Indeed, assume that $\tilde{b}, \tilde{h}$ and $\tilde{\sigma}$ only satisfy (B2). For some fixed $\bar{t}>0$, we define
$\bar{b}(s,x,u)=\tilde{b}(x,u)\mathbf{1}_{[0,\bar{t})}(s)+b(x,u)\mathbf{1}_{[\bar{t},\infty)}(s)$. Similarly, we can define $\bar{h}$ and $\bar{\sigma}$.
Let $(\bar{X},\bar{Y},\bar{Z},\bar{K})$ be the solution to $G$-FBSDE \eqref{App1y}-\eqref{App1m} with generators $(\tilde{b},\tilde{h},\tilde{\sigma},f,g)$.
 Then we have the following result.
\begin{lemma}\label{myq7}
Let $\xi,\xi^{\prime}\in L_{G}^{p}(\Omega
_{t};\mathbb{R}^{n})$ with $p>2$ and $u,u^{\prime}\in\mathcal{U}[t,\infty)$. Then there is a constant ${C}^{\prime}_{\eta}$ depending on $G,\alpha_1,\alpha_2, L, \bar{t},U$ and $\eta$ such that for each $s\geq t$ q.s.
\begin{description}
\item[(i)] $\mathbb{\hat{E}}^{\tilde{G}}_t[|\bar{X}^{t,\xi,u}_s|^2\Gamma^t_s]\leq {C}^{\prime}_{\eta}(1+|\xi|^2);$
\item[(ii)] $\mathbb{\hat{E}}^{\tilde{G}}_t[|\bar{X}^{t,\xi,u}_s-\bar{X}^{t,\xi^{\prime},u}_s|^2\Gamma^t_s]\leq {C}^{\prime}_{\eta}\exp(-2\bar{\eta} (s-t))|\xi-\xi^{\prime}|^2;$
\item[(iii)] $|\bar{X}^{t,\xi,u}_s-\bar{X}^{t,\xi,u^{\prime}}_s|^2\Gamma^t_s\leq \tilde{C}_sM_s+{C}^{\prime}_{\eta}\int^s_t\exp(\bar{\eta}(r-s))| u_r-u^{\prime}_r|^2 \Gamma^t_rdr,
$ where  $M$ is a symmetric $\tilde{G}$-martingale and $\tilde{C}_s$ is a deterministic process.
\end{description}
\end{lemma}
\begin{proof} Without loss of generality, assume $d=1$ and we shall give the sketch of the proof.
To simplify presentation, we shall prove only the case when $t\leq \bar{t}\leq s$.  Then applying Lemma \ref{HW2} on interval $[\bar{t},s]$, we obtain that
\[
\mathbb{\hat{E}}^{\tilde{G}}_{\bar{t}}[|\bar{X}^{t,\xi,u}_s|^2\Gamma^t_s(\Gamma^t_{\bar{t}})^{-1}]=\mathbb{\hat{E}}^{\tilde{G}}_{\bar{t}}[|\bar{X}^{\bar{t},\bar{X}^{t,\xi,u}_{\bar{t}},u}_s|^2\Gamma^t_s(\Gamma^t_{\bar{t}})^{-1}]\leq C_{\eta}(1+|\bar{X}^{t,\xi,u}_{\bar{t}}|^2),
\]
which implies that
\begin{align*}
\mathbb{\hat{E}}^{\tilde{G}}_t[|\bar{X}^{t,\xi,u}_s|^2\Gamma^t_s]\leq C_{\eta} \mathbb{\hat{E}}^{\tilde{G}}_t[\Gamma^t_{\bar{t}}(1+|\bar{X}^{t,\xi,u}_{\bar{t}}|^2)].
\end{align*}
Thus from H\"{o}lder inequality,  Lemmas \ref{myq.1} and \ref{HW212}, we can find a constant $C_{\bar{t}}$ depending on $G,\alpha_1,\alpha_2, L, \bar{t},U$ and $\eta$ such that
\begin{align*}
\mathbb{\hat{E}}^{\tilde{G}}_t[|\bar{X}^{t,\xi,u}_s|^2\Gamma^t_s]\leq
C_{\bar{t}}(1+|\xi|^2),
\end{align*}
and we obtain (i) holds. The property (ii) can be proved in a similar way.

Next we shall prove inequality (iii).  Note that $$
|\bar{X}^{t,\xi,u}_s-\bar{X}^{t,\xi,u^{\prime}}_s|^2\Gamma^t_s\leq 2|\bar{X}^{\bar{t},\bar{X}^{t,\xi,u}_{\bar{t}},u}_s-
\bar{X}^{\bar{t},\bar{X}^{t,\xi,u}_{\bar{t}},u^{\prime}}_s|^2\Gamma^t_s+2|\bar{X}^{\bar{t},\bar{X}^{t,\xi,u}_{\bar{t}},u^{\prime}}_s-\bar{X}^{\bar{t},\bar{X}^{t,\xi,u^{\prime}}_{\bar{t}}, u^{\prime}}_s|^2\Gamma^t_s.
$$
Then applying Lemma \ref{HW2}(ii), we conclude that
\begin{align*}
&|\bar{X}^{\bar{t},\bar{X}^{t,\xi,u}_{\bar{t}},u}_s-
\bar{X}^{\bar{t},\bar{X}^{t,\xi,u}_{\bar{t}},u^{\prime}}_s|^2\Gamma^t_s\\ & \ \ \leq \exp(-\bar{\eta}(s-\bar{t}))M^1_s+C_{\eta}\int^s_{\bar{t}}\exp(-\bar{\eta}(s-r))|u_r-u^{\prime}_r|^2\Gamma^t_rdr,
\end{align*}
where $M^1$ is a symmetric $\tilde{G}$-martingale.
Using the same method as Lemma 4.1 in \cite{HW}, we deduce that
\begin{align*}
&|\bar{X}^{\bar{t},\bar{X}^{t,\xi,u}_{\bar{t}},u^{\prime}}_s-\bar{X}^{\bar{t},\bar{X}^{t,\xi,u^{\prime}}_{\bar{t}}, u^{\prime}}_s|^2\Gamma^t_s
\leq \exp(-\bar{\eta}(s-\bar{t}))(M^2_s+|\bar{X}^{t,\xi,u}_{\bar{t}}-\bar{X}^{t,\xi,u^{\prime}}_{\bar{t}}|^2\Gamma^t_{\bar{t}}),
\end{align*}where $M^2$ is a symmetric $\tilde{G}$-martingale.
Applying $G$-It\^{o} formula and by a similar analysis as in Lemma \ref{HW2}, we can find a constant $C^{\prime}_{\bar{t}}$ depending on $G,\alpha_1,\alpha_2, L, \bar{t},U$ and $\eta$ such that, for each $r\in [t,\bar{t}]$
\[
|\bar{X}^{t,\xi,u}_{r}-\bar{X}^{t,\xi,u^{\prime}}_{r}|^2\Gamma^t_r\leq \tilde{C}^3_rM^3_r+C^{\prime}_{\bar{t}}\int^{\bar{t}}_t|u_l-u^{\prime}_l|^2\Gamma^t_ldl,
\] where $M^3$ is a symmetric $\tilde{G}$-martingale  and $\tilde{C}^3_r$ is a deterministic process.
From these inequalities, one can easily get the desired result.
\end{proof}

\begin{theorem}
\label{le-dpp-3}
Assume that $\xi,\xi^{\prime}\in L_{G}^{p}%
(\Omega_{t};\mathbb{R}^{n})$ with $p>2$ and $u,u^{\prime}\in\mathcal{U}[t,\infty)$.  Then  there exist two constants $q>1$ and  $\bar{C}_{\eta}$ depending only on
$G,U,\eta,L,\alpha_1,\alpha_2,\mu$ and $q$, such that for each $s\geq t$ q.s.
\begin{description}
\item[(i)]   $|Y^{t,\xi,u}_s|\leq \bar{C}_{\eta}(1+|X^{t,\xi,u}_s|^2);$
\item[(ii)] $|Y^{t,\xi,u}_t-Y^{t,\xi^{\prime},u}_t|\leq \bar{C}_{\eta}(1+|\xi|+|\xi^{\prime}|)|\xi-\xi^{\prime}|;$
\item[(iii)] $|Y^{t,\xi,u}_t-Y^{t,\xi,u^{\prime}}_t|\leq \bar{C}_{\eta}(1+|\xi|)|\mathbb{\hat{E}}_t[\int^{\infty}_t\exp(-\mu(r-t))| u_r-u^{\prime}_r|^{2q}dr]|^{\frac{1}{2q}}.$
\end{description}
\end{theorem}
\begin{proof}
The property (i) is immediate from Theorem \ref{HM3}.
Next we shall show the property (iii), since (ii) can be proved in a similar way (see also Lemma A.1 of \cite{HW}).
Without loss of generality, assume that $d=1$. For convenience, we omit superscripts $t$ and $\xi$.

Set $(\hat{Y},\hat{Z})=(Y^{u}-Y^{u^{\prime}},Z^{u}-Z^{u^{\prime}})$.
Then we have for each $s\geq t$,
\[
\hat{Y}_{s}+K_{s}^{u^{\prime}}=\hat{Y}_T+K_{T}^{u^{\prime}}+\int_{s}^{T}\hat{f}_{r}dr+\int%
_{s}^{T}\hat{g}_{r}d\langle B\rangle_{r}-\int_{s}^{T}\hat{Z}_{r}dB_{r}%
-(K_{T}^{u}-K_{s}^{u}),
\]
where $\hat{f}_{s}=f(X^{u}_s,Y_s^{u},Z_{s}%
^{u},u_s)-f(X^{u^{\prime}}_s,Y_s^{u^{\prime}},Z_{s}%
^{u^{\prime}},u^{\prime}_s)$, $\hat{g}_{s}=g(X^{u}_s,Y_s^{u}, Z_{s}%
^{u},u_s)-g(X^{u^{\prime}}_s,Y_s^{u^{\prime}},Z_{s}%
^{u^{\prime}},u^{\prime}_s)$.

By Lemma 3.5  in \cite{HW}, for each $\varepsilon>0$, there exist four bounded processes $a^{\varepsilon}_s, b^{\varepsilon}_s$, $c^{\varepsilon}_s$, $d^{\varepsilon}_s\in M^2_G(0,T)$ for each $T\geq 0$, such that
\[
\hat{f}_{s}=a_{s}^{\varepsilon}\hat{Y}_{s}+b_{s}^{\varepsilon}\hat{Z}%
_{s}+m_s-m_{s}^{\varepsilon},\ \hat{g}_{s}=c_{s}^{\varepsilon}\hat{Y}%
_{s}+d_{s}^{\varepsilon}\hat{Z}_{s}+n_s-n_{s}^{\varepsilon},
\]
and $|b_{s}^{\varepsilon}|\leq \alpha_2$, $|d_{s}^{\varepsilon}|\leq \alpha_2$ $|m_{s}^{\varepsilon}|\leq2(L+\alpha_2)\varepsilon$, $|n_{s}^{\varepsilon}%
|\leq2(L+\alpha_2)\varepsilon$, $a_{s}^{\varepsilon}+2G(c_{s}^{\varepsilon})\leq -\mu$, $m_{s}=f(X^{u}_s,Y_s^{u^{\prime}}, Z_{s}%
^{u^{\prime}},u_s)-f(X^{u^{\prime}}_s,Y_s^{u^{\prime}},Z_{s}%
^{u^{\prime}},u^{\prime}_s)$, $n_{s}=g(X^{u}_s,Y_s^{u^{\prime}},Z_{s}%
^{u^{\prime}},u_s)-g(X^{u^{\prime}}_s,Y_s^{u^{\prime}},Z_{s}%
^{u^{\prime}},u^{\prime}_s)$.

Applying Lemma \ref{the5.2} (see also Theorem 3.6 in \cite{HW}) yields that
\begin{align*}
 \hat{Y}_{t}+K_{t}^{u^{\prime}}
&  =\mathbb{\hat{E}}_{t}^{\tilde{G}}[\tilde{\Gamma}_{T}%
^{t,\varepsilon}(\hat{Y}_{T}+K_{T}^{u^{\prime}})+\int_{t}^{T}(m_s+2G(n_s)-m_{s}^{\varepsilon
}-a_{s}^{\varepsilon}K_{s}^{u^{\prime}})\tilde{\Gamma}_{s}^{t,\varepsilon}ds\\
&  +\int_{t}^{T}(-n_{s}^{\varepsilon}-c_{s}^{\varepsilon}K_{s}^{u^{\prime}}%
)\tilde{\Gamma}_{s}^{t,\varepsilon}d\langle B\rangle_{s}+\int_{t}^{T}n_s\tilde{\Gamma}_{s}^{t,\varepsilon}d\langle B\rangle_s-\int_{t}^{T}2G(n_s)\tilde{\Gamma}_{s}^{t,\varepsilon}ds],
\end{align*}
where $\{\tilde{\Gamma}^{t,\varepsilon}_{s}\}_{s\in\lbrack t,\infty)}$ is given by
\begin{equation*}
\tilde{\Gamma}^{t,\varepsilon}_{s}=\exp(\int_{t}^{s}(a^{\varepsilon}_{r}-b^{\varepsilon}_{r}d^{\varepsilon}_{r})dr+\int_{t}^{s}c^{\varepsilon}_{r}d\langle
B\rangle_{r})\mathcal{E}_{s}^{B}\mathcal{E}_{s}^{\tilde{B}}.
\end{equation*}
Here $\mathcal{E}_{s}^{B}=\exp(\int^s_td^{\varepsilon}_{r}dB_r-\frac{1}{2}\int^s_t|d^{\varepsilon}_{r}|^2d\langle B\rangle_r)$ and
$\mathcal{E}_{s}^{\tilde{B}}=\exp(\int^s_tb^{\varepsilon}_{r}d\tilde{B}_r-\frac{1}{2}\int^s_t|b^{\varepsilon}_{r}|^2d\langle \tilde{B}\rangle_r)$.
Therefore, from equation \eqref{yw2} we get that
\begin{align}
 \hat{Y}_{t}+K_{t}^{u^{\prime}}
 &  \leq \mathbb{\hat{E}}_{t}^{\tilde{G}}%
[\tilde{\Gamma}_{T}^{t,\varepsilon}\hat{Y}_{T}+\int_{t}^{T}(m_s+2G(n_s))\tilde{\Gamma}_{s}^{t,\varepsilon}ds-\int_{t}^{T}m_{s}^{\varepsilon}\tilde{\Gamma}_{s}^{t,\varepsilon}ds\nonumber\\& \ \ \
\ \ \ \ \ \ \ \ \ \ \ \ \ \ \ \ \ -\int_{t}^{T}n_{s}^{\varepsilon
}\tilde{\Gamma}_{s}^{t,\varepsilon}d\langle B\rangle_{s}]+K^{u^{\prime}}_t,\ \ q.s.\label{myq4}
\end{align}
Note that for each $s\geq t$, $
\tilde{\Gamma}^{t,\varepsilon}_s\leq \exp(-\mu (s-t)){\Gamma}^{t,\varepsilon}_s,
$
where $
{\Gamma}^{t,\varepsilon}_{s}=1+\int_{t}^{s}d^{\varepsilon}_{r}{\Gamma}^{t,\varepsilon}_{r}dB_{r}+\int_{t}%
^{s}b^{\varepsilon}_{r}{\Gamma}^{t,\varepsilon}_{r}d\tilde{B}_{r}.
$
Thus it follows from property (i) and Lemma \ref{HW2} (i) that
\begin{align}\label{myq111}
\mathbb{\hat{E}}_{t}^{\tilde{G}}%
[\tilde{\Gamma}_{T}^{t,\varepsilon}\hat{Y}_{T}]\leq 2\exp(-\mu (T-t)) \bar{C}_{\eta}(1+C_{\eta})(1+|\xi|^2) .
\end{align}

Note that $|m_s|+2G(|n_s|)\leq(1+\bar{\sigma}^2)L((1+|X^{u}_s|+|X^{u^{\prime}}_s|)|X^{u}_s-X^{u^{\prime}}_s|+|u_s-u^{\prime}_s|)$.
Then by equation \eqref{myq4},
we derive that
\begin{align}
\hat{Y}_{t}\leq &
(1+\bar{\sigma}^2)L\mathbb{\hat{E}}^{\tilde{G}}_t[\int_{t}^{T}\exp(-\mu (s-t))((1+|X^{u}_s|+|X^{u^{\prime}}_s|)|X^{u}_s-X^{u^{\prime}}_s|+|u_s-u^{\prime}_s|){\Gamma}^{t,\varepsilon}_{s}ds]\nonumber\\ & \ \ \ \
+2\exp(-\mu (T-t))  \bar{C}_{\eta}(1+C_{\eta})(1+|\xi|^2)+\frac{2(L+\alpha_2)(1+\bar{\sigma}^2)}{\mu}{\varepsilon}.\label{myq5}
\end{align}

Recalling H\"{o}lder's inequality and Lemma \ref{HW2} (i), we conclude that
\begin{align*}
&\mathbb{\hat{E}}^{\tilde{G}}_t[\int_{t}^{T}\exp(-\mu (s-t))(1+|X^{u}_s|+|X^{u^{\prime}}_s|)|X^{u}_s-X^{u^{\prime}}_s|{\Gamma}^{t,\varepsilon}_{s}ds]\\
&\leq \sqrt{3}\mathbb{\hat{E}}^{\tilde{G}}_t[\int_{t}^{T}\exp(-\mu (s-t))(1+|X^{u}_s|^2+|X^{u^{\prime}}_s|^2){\Gamma}^{t,\varepsilon}_{s}ds]^{\frac{1}{2}}\mathbb{\hat{E}}^{\tilde{G}}_t[\int_{t}^{T}\exp(-\mu (s-t))|X^{u}_s-X^{u^{\prime}}_s|^2{\Gamma}^{t,\varepsilon}_{s}ds]^{\frac{1}{2}}\\
&\leq (\frac{6(1+{C_{\eta}})}{\mu})^{\frac{1}{2}}(1+|\xi|)\mathbb{\hat{E}}^{\tilde{G}}_t[\int_{t}^{T}\exp(-\mu (s-t))|X^{u}_s-X^{u^{\prime}}_s|^2{\Gamma}^{t,\varepsilon}_{s}ds]^{\frac{1}{2}}.
\end{align*}

On the other hand, recalling  Lemma \ref{HW2} (iii), we get that
\begin{align*}
& \mathbb{\hat{E}}^{\tilde{G}}_t[\int_{t}^{T}\exp(-\mu (s-t))|X^{u}_s-X^{u^{\prime}}_s|^2{\Gamma}^{t,\varepsilon}_{s}ds]^{\frac{1}{2}}
\\&\leq\mathbb{\hat{E}}^{\tilde{G}}_t[\int_{t}^{T}\exp(-\mu (s-t))[\exp(-\bar{\eta} (s-t))M_s
 +C_{\eta}\int^s_t\exp(\bar{\eta}(r-s))|u_r-u^{\prime}_r|^2\Gamma_{r}^{t,\varepsilon}dr]ds]\}^{\frac{1}{2}}\\&
\leq\mathbb{\hat{E}}^{\tilde{G}}_t[\int_{t}^{T}C_{\eta}\exp(-\mu (s-t))\int^s_t\exp(\bar{\eta}(r-s))|u_r-u^{\prime}_r|^2{\Gamma}^{t,\varepsilon}_rdr]ds]^{\frac{1}{2}}
\\&=\sqrt{C_{\eta}}\mathbb{\hat{E}}^{\tilde{G}}_t[\int_{t}^{T}\exp(-\mu (r-t))|u_r-u^{\prime}_r|^2{\Gamma}^{t,\varepsilon}_rdr]^{\frac{1}{2}},
\end{align*}
where we have used Fubini's theorem in the last equality. Then by a similar analysis,
we can also obtain that
\begin{align*}
\mathbb{\hat{E}}^{\tilde{G}}_t[\int_{t}^{T}\exp(-\mu(s-t))|u_s-u^{\prime}_s|{\Gamma}^{t,\varepsilon}_sds]
\leq (\frac{1}{\mu})^{\frac{1}{2}}\{\mathbb{\hat{E}}^{\tilde{G}}_t[\int_{t}^{T}\exp(-\mu(r-t))|u_r-u^{\prime}_r|^2{\Gamma}^{t,\varepsilon}_rdr]\}^{\frac{1}{2}}.
\end{align*}
Applying H\"{o}lder's inequality again, we deduce that
\begin{align*}
&\mathbb{\hat{E}}^{\tilde{G}}_t[\int_{t}^{T}\exp(-\mu (r-t))|u_r-u^{\prime}_r|^2{\Gamma}^{t,\varepsilon}_rdr]
\\&\leq\mathbb{\hat{E}}^{\tilde{G}}_t[\int_{t}^{T}\exp(-\mu  (r-t))|u_r-u^{\prime}_r|^{2q}dr]^{\frac{1}{q}}|\int_{t}^{T}\exp(-\mu (r-t))\mathbb{\hat{E}}^{\tilde{G}}_t[|{\Gamma}^{t,\varepsilon}_r|^p]dr|^{\frac{1}{p}},
\end{align*}
where $1/q+1/p=1$. Then by Lemma \ref{HW212} and choosing $p\in (1,2)$ small enough, there exists a constant $C_{\mu}$ depending on $\mu$ and $p$, such that
\begin{align*}
&\mathbb{\hat{E}}^{\tilde{G}}_t[\int_{t}^{T}\exp(-\mu (r-t))|u_r-u^{\prime}_r|^2{\Gamma}^{t,\varepsilon}_rdr]
\leq C_{\mu}\mathbb{\hat{E}}^{\tilde{G}}_t[\int_{t}^{T}\exp(-\mu (r-t))|u_r-u^{\prime}_r|^{2q}dr]^{\frac{1}{q}}.
\end{align*}

Therefore, by equation \eqref{myq5}, sending $\varepsilon\rightarrow 0$ and then letting $T\rightarrow\infty$,
 we could find a constant $\tilde{C}_{\eta}$ depending only on
$G,U,\eta,L,\alpha_1,\alpha_2,q$ and $\mu$ so that
\begin{align*}
\hat{Y}_{t}\leq \tilde{C}_{\eta}(1+|\xi|)\mathbb{\hat{E}}^{\tilde{G}}_t[\int_{t}^{\infty}\exp(-\mu (r-t))|u_r-u^{\prime}_r|^{2q}dr]^{\frac{1}{2q}}.
\end{align*}
Using the same method, we  also have that
\begin{align*}
Y^{u^{\prime}}_t-Y^{u}_t\leq \tilde{C}_{\eta}(1+|\xi|)\mathbb{\hat{E}}^{\tilde{G}}_t[\int_{t}^{\infty}\exp(-\mu (r-t))|u_r-u^{\prime}_r|^{2q}dr]^{\frac{1}{2q}}.
\end{align*}
which is the desired result.
\end{proof}
\begin{remark}
{\upshape
We remark that the above lemma also holds for $\bar{Y}$ by Lemma \ref{myq7}.
}
\end{remark}

Now we shall give the main results of this section.
\begin{lemma}
\label{le-dpp-2} Let $u\in\mathcal{U}[t,\infty)$ be given. Then there exists a
sequence $(u^{k})_{k\geq1}$ in $\mathbb{U}[t,\infty)$ such that%
\[
\lim_{k\rightarrow\infty}\mathbb{\hat{E}}[\int_{t}^{\infty}\exp(-\mu s)|u_{s}-u_{s}^{k}%
|^{2q}ds]=0.
\]
\end{lemma}
\begin{proof}
Note that $u$ is  bounded by $M:=\sup\{|a|:a\in U\}$. Then for each $\varepsilon>0$, there is a constant $T$ such that
\[
\int^{\infty}_T\exp(-\mu s)ds\leq\frac{\varepsilon}{2^{4q+1}M^{2q}}.
\]
By Remark \ref{myq6} and using the same method as in Lemma 13 in \cite{HJ1}, we can find a process
$v^{\prime}\in\mathbb{U}[t,T]$ such that $$\mathbb{\hat{E}}[\int_{t}^{T}|u_{s}%
-v^{\prime}_{s}|^{2q}ds]\leq\frac{\varepsilon}{2^{2q}}.$$
Denote $v_s:=v^{\prime}_s\mathbf{1}_{[t,T]}(s)+u_0\mathbf{1}_{(T,\infty)}(s),$ where $u_0\in U$ is a fixed constant.
It is easy to check that $v\in\mathbb{U}[t,\infty).$

Then we have
\begin{align*}
&\mathbb{\hat{E}}[\int_{t}^{\infty}\exp(-\mu  s)|u_{s}-v_{s}|^{2q}ds]\\&
\leq 2^{2q-1}\mathbb{\hat{E}}[\int_{t}^{\infty}\exp(-\mu s)|u_{s}-v^{\prime}_{s}|^{2q}ds]+2^{2q-1}\mathbb{\hat{E}}[\int_{t}^{\infty}\exp(-\mu s)|v_{s}-v^{\prime}_{s}|^{2q}ds]
\\&\leq \frac{\varepsilon}{2}+2^{2q-1}\mathbb{\hat{E}}[\int_{T}^{\infty}\exp(-\mu s)|u_{s}-v^{\prime}_{s}|^{2q}ds]+2^{2q-1}\mathbb{\hat{E}}[\int_{T}^{\infty}\exp(-\mu s)|v_{s}-v^{\prime}_{s}|^{2q}ds]\\&\leq \varepsilon,
\end{align*}
which completes the proof.
\end{proof}

\begin{theorem}
\label{the-dpp-11} The value function $V(t,x)$ is a deterministic function  and
\[
V(t,x)=\inf_{u\in\mathcal{U}^{t}[t,\infty)}Y_{t}^{t,x,u}.
\]
Moreover, $V(x)=V(t,x)$ for each $t\geq 0$.
\end{theorem}
\begin{proof}
Note that  $Y_{t}^{t,x,u}$
is a constant for each $u\in\mathcal{U}^{t}[t,\infty)$.
Since $\mathcal{U}^{t}[t,\infty)\subset\mathcal{U}[t,\infty)$, it is easy to check that
\[
\inf_{u\in\mathcal{U}^{t}[t,\infty)}Y_{t}^{t,x,u}\geq \underset{u\in\mathcal{U}[t,\infty)}{ess\inf}Y_{t}^{t,x,u}.
\]
In the following we shall
show that $Y_{t}^{t,x,u}\geq\inf_{v\in\mathcal{U}^{t}[t,\infty)}Y_{t}^{t,x,v}$ q.s.
for each $u\in\mathcal{U}[t,\infty)$.

For each given $u\in\mathcal{U}[t,\infty),$ from
Lemma \ref{le-dpp-2}, we can find  a sequence $u^{k}=\sum_{i=1}^{N_{k}%
}\mathbf{1}_{A_{i,k}}u^{i,k}\in\mathbb{U}[t,\infty)$, $k=1,2,...,$ such that
\[
\lim\limits_{k\rightarrow\infty}\mathbb{\hat{E}}[\int_{t}^{\infty}\exp(-\mu s)|u_{s}-u^k_s|^{2q}ds]=0.
\]

By the uniqueness  of $G$-FBSDE with infinite horizon and the standard arguments, we can obtain that
\[
\sum_{i=1}^{N_{k}}\mathbf{1}_{A_{i,k}}%
Y_{t}^{t,x,u^{i,k}}=Y_t^{t,x,u^k} \ \ { q.s.}
\]
Then applying Theorem \ref{le-dpp-3} (iii) and choosing a subsequence if
necessary, we deduce that  $\sum_{i=1}^{N_{k}}\mathbf{1}_{A_{i,k}}Y_{t}^{t,x,u^{i,k}}$ converges
to $Y_{t}^{t,x,u}$ q.s.
Therefore, it follows from
\[
\sum_{i=1}^{N_{k}}\mathbf{1}_{A_{i,k}}Y_{t}^{t,x,u^{i,k}}\geq\inf_{v\in
\mathcal{U}^{t}[t,T]}Y_{t}^{t,x,v}\ \ { q.s.}
\]
 that $Y_{t}^{t,x,u}\geq\inf_{v\in\mathcal{U}^{t}[t,\infty)}Y_{t}^{t,x,v}$
q.s. Thus
\[
V(t,x)=\underset{u\in\mathcal{U}[t,\infty)}{ess\inf}Y_{t}^{t,x,u}=\inf
_{v\in\mathcal{U}^{t}[t,\infty)}Y_{t}^{t,x,v}.
\]
Note that $\{B_{s+t}-B_t\}_{s\geq 0}$ is also a $G$-Brownian motion and $\mathcal{U}^t[t,\infty)$ is the shifted space with respect to $\mathcal{U}[0,\infty)$.
Then by the uniqueness of $G$-BSDEs with infinite horizon we  get $V(t,x)=V(0,x)$ for each $t\geq 0$ and this completes the proof.
\end{proof}

\begin{corollary}
\label{le-dpp-4}
 For any $x,y\in
\mathbb{R}^{n}$, we have
\[
 |V(x)-V(y)|\leq \bar{C}_{\eta} (1+|x|+|y|)|x-y|.
\]

\end{corollary}
\begin{proof}
The proof is immediate from Theorems \ref{le-dpp-3} and  \ref{the-dpp-11}.
\end{proof}

\begin{theorem}
\label{the-dpp-1-1} For each $\xi\in L_{G}^{p}(\Omega_{t};\mathbb{R}^{n})$ with $p>2$, we have%
\[
V(\xi)=\underset{u\in\mathcal{U}[t, \infty)}{ess\inf}Y_{t}^{t,\xi,u} \ \ q.s.
\]
\end{theorem}
\begin{proof}
The proof is similar to the one in \cite{HJ1}.
For readers' convenience, we shall give the sketch of the proof.

First, we claim that $V(\xi)\leq Y_{t}^{t,\xi,u}$ q.s. for each $u\in\mathcal{U}[t,\infty)$. Indeed, for any $\xi\in L_{G}^{p}(\Omega_{t};\mathbb{R}^{n})$ with $p>2$, there is  a sequence $\xi^{k}=\sum_{i=1}^{N_{k}}%
x_{i,k}\mathbf{1}_{A_{i,k}}$, $k=1,2,...,$ such that $\lim_{k\rightarrow\infty}|\xi-\xi^{k}|=0$ q.s. and
$
\lim_{k\rightarrow\infty}\mathbb{\hat{E}}[|\xi-\xi^{k}|^2]=0,
$
where $x_{i,k}\in\mathbb{R}^{n}$ and
$\{A_{i,k}\}_{i=1}^{N_{k}}$ is a $\mathcal{B}(\Omega_{t})$-partition of
$\Omega$.  By Corollary \ref{le-dpp-4}, we have%
\[
|V(\xi)-V(\xi^{k})|\leq \bar{C}_{\eta}(1+|\xi|+|\xi^k|)|\xi-\xi^{k}|.
\]
Recalling Lemma \ref{le-dpp-3}, we derive that
\begin{align}\label{myq121}
|Y_{t}^{t,\xi,u}-\sum_{i=1}^{N_{k}}\mathbf{1}_{A_{i,k}}Y_{t}^{t,x_{i,k},u}|
= \sum_{i=1}^{N_{k}}|Y_{t}^{t,\xi,u}-Y_{t}^{t,x_{i,k},u}|\mathbf{1}_{A_{i,k}}
\leq \bar{C}_{\eta}(1+|\xi|+|\xi^k|)|\xi-\xi^{k}|.
\end{align}
Note that \[
V(\xi^{k})=\sum_{i=1}^{N_{k}}\mathbf{1}_{A_{i,k}}V(x_{i,k})\leq \sum_{i=1}^{N_{k}}\mathbf{1}_{A_{i,k}}Y_{t}^{t,x_{i,k},u} \ \  q.s.
\]Consequently, sending $k\rightarrow\infty$ yields the desired result.

Next, suppose  $\eta\in L_{G}^{2}(\Omega_{t})$ satisfies that  $\eta\leq Y_{t}^{t,\xi,u}$  q.s. for each  $u\in\mathcal{U}[t,\infty)$.  Then it suffices to show that $\eta\leq
V(t,\xi)$ q.s.
By equation \eqref{myq121}, we deduce that for any $u\in\mathcal{U}[t,\infty)$,%
\[
\eta\leq\sum_{i=1}^{N_{k}}\mathbf{1}_{A_{i,k}}Y_{t}^{t,x_{i,k},u}+ \bar{C}_{\eta}(1+|\xi|+|\xi^k|)|\xi-\xi^{k}|\ \  q.s.,
\]
which together with Theorem \ref{the-dpp-11} indicate that for each $k$
\begin{align*}
\eta \leq V(t,\xi^{k})+ \bar{C}_{\eta}(1+|\xi|+|\xi^k|)|\xi-\xi^{k}| \ \ q.s.
\end{align*}
Letting $k\rightarrow\infty$, we obtain that $\eta\leq V(\xi)$ q.s. The proof is complete.
\end{proof}
\begin{remark}\label{myq8}
{\upshape
We remark that the above results also remain true for the stochastic control problem associated with
$\bar{Y}$. However,  the value function $\bar{V}$ depends on time variable $t$ in this case. Indeed, we have
$
\bar{V}(t,\xi)=\underset{u\in\mathcal{U}[t, \infty)}{ess\inf}\bar{Y}_{t}^{t,\xi,u} q.s.
$
}
\end{remark}
\section{Dynamic programming principle and related HJBI equation}
In this section, we shall establish the link between the value function $V$ and the corresponding
HJBI equation. The main tool is the
stochastic ``backward semigroup'' introduced by Peng  \cite{peng-dpp-1}.

For each $(t,x)\in [0,\infty)\times\mathbb{R}^n$,  positive real number $\delta$,
 $u\in\mathcal{U}[t,t+\delta]$ and $\eta\in L_{G}^{p}(\Omega_{t+\delta})$ with $p>2$, we
define the following backward semigroups:
\[
\mathbb{G}_{t,t+\delta}^{t,x,u}[\eta]:=\tilde{Y}_{t}^{t,t+\delta,x,u},
\]
where $(X_{s}^{t,x,u},\tilde{Y}_{s}^{t,t+\delta,x,u},\tilde{Z}_{s}^{t,t+\delta,x,u},\tilde{K}_{s}^{t,t+\delta,x,u})_{t\leq
s\leq t+\delta}$ is the solution of  the following type of $G$-FBSDEs in the interval $[t,t+\delta]$:%

\begin{align} \label{App1}
\begin{cases}
&X_{s}^{t,x,u}=x+\int^s_tb(X_{r}^{t,x,u},u_r)dr+\int^s_th_{ij}(X_{r}^{t,x,u},u_r)d\langle
B^i,B^j\rangle_{r}+\int^s_t\sigma(X_{r}^{t,x,u},u_r)dB_{r},\\&
\tilde{Y}_{s}^{t,t+\delta,x,u}   =\eta+\int_{s}^{t+\delta}f(X_{r}^{t,x,u}%
,\tilde{Y}_{r}^{t,t+\delta,x,u},\tilde{Z}_{r}^{t,t+\delta,x,u},u_r)dr-\int_{s}^{t+\delta}\tilde{Z}_{r}^{t,t+\delta,x,u}dB_{r}\\
& \ \ \ \ \ \ \ \ \ \ \ \ \ \ \ \ \ \ +\int_{s}^{t+\delta}g_{ij}(X_{r}^{t,t+\delta,x,u}%
,\tilde{Y}_{r}^{t,t+\delta,x,u},\tilde{Z}_{r}^{t,t+\delta,x,u},u_r)d\langle B^i,B^j\rangle_{r} \\ & \ \ \ \ \ \ \ \ \ \ \ \ \ \ \ \ \ \ -(\tilde{K}_{t+\delta}^{t,t+\delta,x,u}-\tilde{K}_{s}^{t,t+\delta,x,u}).
\end{cases}
\end{align}

Then we have the following dynamic programming principle.
\begin{theorem}
\label{dpp-2}   Assume \emph{(B1)-(B5)} hold.
Then for each $s>0$ and $x\in\mathbb{R}^{n}$, we have
\begin{equation}%
V(x)=  \underset{u\in\mathcal{U}[0,s]}{\inf}%
\mathbb{G}_{0,s}^{0,x,u}[V(X_{s}^{x,u})].
\label{DPP}%
\end{equation}
\end{theorem}

In order to prove it, we need the following lemma.
\begin{lemma}
\label{le-new-dpp2} Assume \emph{(B1)-(B5)} hold. Then for any $s>0$ and $x\in\mathbb{R}^{n}$, the following inequality holds true:%
\[
V(x)\leq\underset{u\in\mathcal{U}[0,s]}{\inf}\mathbb{G}%
_{0,s}^{0,x,u}[V(X_{s}^{x,u})].
\]
\end{lemma}
\begin{proof} The proof will be divided into the following two steps.

{\bf Step 1:}
For each fixed $N>0$, we set $b^{i_{1},N}=(b^{i_{1}}\wedge N)\vee(-N)$,
$h_{ij}^{i_{1},N}=(h_{ij}^{i_{1}}\wedge N)\vee(-N)$, $\sigma_{i_{1}i_{2}}%
^{N}=(\sigma_{i_{1}i_{2}}\wedge N)\vee(-N)$ for $i_{1}\leq n$, $i_{2}\leq d$
and $b^{N}=(b^{1,N},\ldots,b^{n,N})^{\top}$, $h_{ij}^{N}=(h_{ij}^{1,N}%
,\ldots,h_{ij}^{n,N})^{\top}$, $\sigma^{N}=(\sigma_{i_{1}i_{2}}^{N})$.  Then we define $\bar{b}^N(t,x,u)=b^N(x,u)\mathbf{1}_{[0,s)}(t)+b(x,u)\mathbf{1}_{[s,\infty)}(t)$. Similarly, we can define $\bar{h}^N$ and
$\bar{\sigma}^N$. Note that in general  ${b}^N, {h}^N $ and ${\sigma}^N$ only satisfy  assumption (B2).

By Remark \ref{myq8}, we derive that
\[
V^{N}(t,x):=\underset{u\in\mathcal{U}[t,\infty)}{ess\inf}Y_{t}^{t,x,u,N}=\underset{u\in\mathcal{U}^{t}[t,\infty)}{\inf}Y_{t}^{t,x,u,N}.\]
Note that $V^N(t,x)=V(x)$ for $t\geq s.$ We claim that for any $s\geq 0$ and $x\in\mathbb{R}^{n}$,%
\begin{equation}
V^{N}(0,x)\leq\underset{u\in\mathcal{U}[0,s]}{\inf}\mathbb{G}%
_{0,s}^{0,x,u,N}[V^{N}(s,X_{s}^{x,u,N})]=\underset{u\in\mathcal{U}[0,s]}{\inf}\mathbb{G}%
_{0,s}^{0,x,u,N}[V(X_{s}^{x,u,N})], \label{eq-dpp-2}%
\end{equation}
where $\mathbb{G}_{t,s}^{t,x,u,N}[\cdot]$ is defined in the same way as $\mathbb{G}%
_{t,s}^{t,x,u}[\cdot]$. The proof will be given in the next step.

Note that there exists a constant $C_{2}>0$ (may vary from line to line) depending on  $L$, $s$, $n$, $U$, $G$, $\alpha_1,$ $\alpha_2$ and $\eta$,  such that for any $u\in\mathcal{U}[0,\infty)$,%
\begin{align*}
|Y_{0}^{x,u,N}-Y_{0}^{x,u}| \leq\frac{C_{2}(1+|x|^{3})}{N}.
\end{align*}
 Indeed,
applying Lemma \ref{myq7},  we get for any
$u\in\mathcal{U}[0,\infty)$, $t\geq s$,
\begin{align*}
\mathbb{\hat{E}}^{\tilde{G}}[|X_{t}^{x,u,N}-X_{t}^{x,u}|^2\Gamma^0_t]  &  \leq \exp(-2\bar{\eta}(t-s))
\mathbb{\hat{E}}^{\tilde{G}}[|X_{s}^{x,u,N}-X_{s}^{x,u}|^2\Gamma^0_s].
\end{align*}
By Lemma \ref{myq.1}, H\"{o}lder's inequality and a standard argument (see, e.g., Lemma 24 in \cite{HJ1}), for each $p\geq 1$, there exists a constant
$C_{s,p}>0$ depending on $s$, $p$, $n$, $U$, $G$, $\alpha_1$ and $L$ such that for any
$u\in\mathcal{U}[0,\infty)$, $r\in\lbrack 0,s]$,
\begin{align}\label{myq200}
\mathbb{\hat{E}}^{\tilde{G}}[|X_{r}^{x,u,N}-X_{r}^{x,u}|^{2p}\Gamma^0_r]  & \leq\frac{C_{s,p}\exp(-2\bar{\eta}r)(1+|x|^{4p})}{N^{2p}}.
\end{align}

Therefore, we obtain that \[
\mathbb{\hat{E}}^{\tilde{G}}[|X_{t}^{x,u,N}-X_{t}^{x,u}|^2\Gamma^0_t]\leq  \frac{C_{s,1} \exp(-2\bar{\eta}t)(1+|x|^{4})}{N^2}.
\]
Then by a similar analysis as in  Lemma \ref{le-dpp-3},
we can obtain  the desired result. Consequently,%
\begin{align*}
|V^{N}(0,x)-V(0,x)|\leq\underset{u\in\mathcal{U}[0,\infty)}{\sup}|Y_{0}^{x,u,N}-Y_{0}^{x,u}|\leq\frac{C_{2}(1+|x|^{3})}{N}.%
\end{align*}

From Theorem \ref{pro3.5},  Corollary \ref{le-dpp-4} and inequality \eqref{myq200}, we have for any
$u\in\mathcal{U}[0,s]$,
\begin{align*}
&  |\mathbb{G}_{0,s}^{0,x,u,N}[V(X_{s}^{x,u,N})]-\mathbb{G}%
_{0,s}^{0,x,u}[V(X_{s}^{x,u})]|^{2}\\
&  \leq C_{2} \mathbb{\hat{E}}[|V(X_{s}^{x,u,N})-V(X_{s}%
^{x,u})|^{2}+\int_{0}^{s}(1+|X_{t}^{x,u,N}|^2+|X_{t}^{x,u}|^2)|X_{t}^{x,u,N}-X_{t}^{x,u}|^{2}dt]\\&
\leq \frac{C_{2}(1+|x|^{6})}{N^{2}} .
\end{align*}
Thus%
\[
|\underset{u\in\mathcal{U}[0,s]}{\inf}\mathbb{G}_{0,s}%
^{0,x,u,N}[V(X_{s}^{x,u,N})]-\underset{u\in\mathcal{U}%
[0,s]}{\inf}\mathbb{G}_{0,s}^{0,x,u}[V(X_{s}^{x,u})]|\leq\frac
{C_{2}(1+|x|^{3})}{N}.
\]
Sending $N\rightarrow
\infty$ in inequality (\ref{eq-dpp-2}), we get the desired result.

{\bf Step 2:} Now we shall complete the proof of equation \eqref{eq-dpp-2}. The main idea is from the Lemma 22 in \cite{HJ1} and we shall only give the sketch of the proof.
For each $\varepsilon>0$, there exists a $u\in\mathcal{U}[0,s]$ such
that%
\begin{align}\label{yw6}
\mathbb{G}_{0,s}^{0,x,u,N}[V(X_{s}^{x,u,N})]-\varepsilon\leq\underset{v\in\mathcal{U}[0,s]}{\inf}\mathbb{G}_{0,s}^{0,x,v,N}[V(X_{s}^{x,v,N})].
\end{align}

Now consider the following SDE: for
any $v\in\mathcal{U}[0,s]$,
\begin{align}
\begin{cases}\label{yw4}
& d\tilde{X}_{r}^{x,v,N}= b^N(\tilde{X}_{r}^{x,v,N}-\tilde{X}^N_{r}%
e,v_{r})dr+h^N_{ij}(\tilde{X}_{r}^{x,v,N}-\tilde{X}_{r}^Ne,v_{r})d\langle
B^{i},B^{j}\rangle_{r}\\
& \ \ \ \ \ \ \ \ \ \  \ \ \ \ \ \ \ \ \ \ +\sigma^N(\tilde{X}_{r}^{x,v,N}-\tilde{X}^N_{r}e,v_{r})dB_{r}+(N
+1)edB_{r}^{1},\\
&d\tilde{X}_{r}^N=  (N+1)dB_{r}^{1},\\
&\tilde{X}_{0}^{x,v,N}=  x,\text{ }\tilde{X}^N_{0}=0,\text{ }r\in\lbrack 0,s],
\end{cases}
\end{align}
where $e=[1,\ldots,1]^{\top}\in\mathbb{R}^{n}$ and $B^1$ is the first component of $G$-Brownian motion $B$. By the uniqueness of $G$-SDE, one can easily check that
\[
\tilde{X}_{r}^{x,v,N}=X_{r}^{x,v,N}+(N+1)B_{r}^{1}e,\text{
}\tilde{X}_{r}^N=(N+1)B_{r}^{1},\text{ }r\in\lbrack 0,s]
\]
is the solution to equation \eqref{yw4}. Note that $b^N, h^N$ and $\sigma^N$ are bounded and \[
\sqrt{|\sigma^N_{i1}+N+1|^{2}+|\sigma_{i2}^N|^{2}+\cdots+|\sigma_{id}^N|^{2}%
}\geq1.
\]
Thus applying Theorem 3.18 in \cite{HW1} yields that $\mathbf{1}_{\{ \tilde{X}_{s}^{x,v,N}\in\lbrack
a,b)\}}\in L_{G}^{2}(\Omega_{s})$ for any $a$, $b
\in\mathbb{R}^{n}$ with $a\leq b$.

Then by the same way as in Lemma 22 in \cite{HJ1} and Lemma \ref{myq.1}, for each $k\geq 1$ we can find a simple function $\xi^{k,u,N}\in L^2_G(\Omega_s)$  and an admissible control $\bar{u}^{k,N}\in\mathbb{U}(s,\infty)$ so that
\begin{align}
\mathbb{\hat{E}}[|X_{s}^{x,u,N}-\xi^{k,u,N}|^{4}]  &  \leq\frac{C_{2,N}(1+|x|^{8})}{k^{4}},\nonumber\\
V(\xi^{k,u,N})\leq Y_{s}^{s,\xi^{k,u,N},\bar{u}^{k,N},N}& \leq V(\xi^{k,u,N}%
)+\varepsilon,\label{yw5}
\end{align}
where $C_{2,N}$ is a constant (may vary from line to line) depending on $x,N,s,G,u, n$, $\alpha_1, \alpha_2, \bar{\eta}$ and $L$.
Consequently, applying Theorem \ref{le-dpp-3} (ii), equation \eqref{yw5} and H\"{o}lder's inequality yields that
\begin{align*}
&\mathbb{\hat{E}}[|Y_{s}^{s,X_{s}^{x,u,N},\bar{u}^{k,N},N}-V(X_{s}^{x,u,N})|^2]\\
& \ \leq 2(\mathbb{\hat{E}}[|Y_{s}^{s,X_{s}^{x,u,N},\bar{u}^{k,N},N}-Y_{s}^{s,\xi^{k,u,N}%
,\bar{u}^{k,N},N}|^2]+
\mathbb{\hat{E}}[|Y_{s}^{s,\xi^{k,u,N}%
,\bar{u}^{k,N},N}-V(X_{s}^{x,u,N})|^2])\\
& \
\leq 2(\bar{C}_{\eta}\mathbb{\hat{E}}[(1+|X_{s}^{x,u,N}|+|\xi^{k,u,N}|)^2|X_{s}^{x,u,N}-\xi^{k,u,N}|^2]+
\mathbb{\hat{E}}[(|V(\xi^{k,u,N})-V(X_{s}^{x,u,N})|+\varepsilon)^2])\\
& \
\leq C_{2,N}\mathbb{\hat{E}}[(1+|X_{s}^{x,u,N}|+|\xi^{k,u,N}|)^4]^{\frac{1}{2}}\mathbb{\hat{E}}[|X_{s}^{x,u,N}-\xi^{k,u,N}|^4]^{\frac{1}{2}}+C_{2,N}\varepsilon^2\\& \
\leq\frac{C_{2,N}(1+|x|^{6})}{k^2}+C_{2,N}\varepsilon^2,
\end{align*}
which together with Theorem \ref{pro3.5} imply that
\begin{align}\label{yw3}
|\mathbb{G}_{0,s}^{0,x,u,N}[Y_{s}^{s,X_{s}^{x,u,N},\bar{u}^{k,N},N}%
]-\mathbb{G}_{0,s}^{0,x,u,N}[V(X_{s}^{x,u,N})]|
\leq\frac{C_{2,N}(1+|x|^{3})}{k}+C_{2,N}\varepsilon.
\end{align}

Then we denote $\tilde{u}^N(r)=u(r)\mathbf{1}_{[0,s]}(r)+\bar{u}^{k,N}(r)\mathbf{1}_{(s,\infty)}(r)$, which belongs to $\mathcal{U}[0,\infty)$. Thus by the
definition of $V^N$,
\[
V^N(0,x)\leq Y_{0}^{x,\tilde{u}^N,N}=\mathbb{G}_{0,s}^{0,x,u,N}[Y_{s}^{s,X_{s}%
^{x,u,N},\bar{u}^{k,N},N}],
\]
which together with equations \eqref{yw6} and \eqref{yw3} imply that
\[
V^N(0,x)-\frac{{C_{2,N}(1+|x|^{3})}}{k}-C_{2,N}\varepsilon
\leq\underset{v\in\mathcal{U}[0,s]}{\inf}\mathbb{G}_{0,s}%
^{0,x,v,N}[V(X_{s}^{x,v,N})].
\]
Sending $k\rightarrow\infty$ and then letting $\varepsilon\rightarrow0$ in the above inequality yield the desired result.
\end{proof}
\begin{remark}{\upshape Note that the method of \cite{HJ1} cannot be directly applied to deal with the  above question, since the set of admissible controls is  more complicated in our setting.
Thus we introduce a new version of ``implied partition''.
}
\end{remark}

Now we are ready to present  the proof of Theorem \ref{dpp-2}.

\begin{proof}[The proof of   Theorem \ref{dpp-2}]
 By Theorem \ref{the-dpp-1-1}, we obtain for any $u\in\mathcal{U}[0,\infty)$,
\[
Y_{s}^{s,X_{s}^{x,u},u}\geq V(X_{s}^{x,u})\; \ {q.s.,}%
\]
where $Y_{s}^{s,X_{s}^{x,u},u}=Y_{s}^{x,u}$ is the solution of equation
\eqref{App1m}. Then, by the comparison theorem of $G$-BSDE,
we obtain that $
Y_{0}^{x,u}\geq\mathbb{G}_{0,s}^{0,x,u}[V(X_{s}^{x,u})] \ {q.s.},
$
which concludes that
\[
V(x)\geq\underset{u\in\mathcal{U}[0,s]}{\inf}%
\mathbb{G}_{0,s}^{0,x,u}[V(X_{s}^{x,u})].
\]
Recalling Lemma \ref{le-new-dpp2}, we can obtain that
\[
V(x)=\underset{u\in\mathcal{U}[0,s]}{\inf}%
\mathbb{G}_{0,s}^{0,x,u}[V(X_{s}^{x,u})],
\]
which ends the proof.
\end{proof}

Next, we shall prove the value function $V$ is the viscosity solution to the related HJBI equation. Note that  $H$ in equation \eqref{feynman} is not uniformly continuous in $(x,p,A)$, which is different from the ones in \cite{CMI} (see also \cite{N1, PE}).
Thus we introduce a probabilistic method to treat the uniqueness problem of viscosity solutions.

\begin{theorem}
\label{viscosity}  Assume \emph{(B1)-(B5)} hold.   Then $V$ is the unique  viscosity
solution  of the following HJBI equation with quadratic growth:
\begin{align}\label{feynman}
\underset{u\in U}{\inf}[G(H(x,V,\partial_{x}%
V,\partial_{xx}^{2}V,u))+\langle \partial_{x}%
V,b(x,u)\rangle+f(x,V,\partial_{x}%
V\sigma(x,u),u)]=0,
\end{align}
where%
\[%
\begin{array}
[c]{cl}%
H_{ij}(x,v,p,A,u)= & (\sigma^{\top}(x,u)A\sigma(x,u))_{ij}+2\langle
p,h_{ij}(x,u)\rangle
 +2g_{ij}(x,v,p\sigma(x,u),u)
\end{array}
\]
for any $(x,v,p,A,u)\in\mathbb{R}^{n}\times\mathbb{R}%
\times\mathbb{R}^{n}\times\mathbb{S}({n})\times U$.
\end{theorem}
\begin{proof}
From Theorem \ref{dpp-2}, we can prove that $V$ is a viscosity solution of equation \eqref{feynman}  in the same way as in Theorem 26 of \cite{HJ1}.
Next we shall give the uniqueness of viscosity solution of equation (\ref{feynman}).

Suppose $\tilde{V}$ is also a viscosity solution of equation \eqref{feynman} with quadratic growth.
 For each $T>0$, it is easy to check that $\tilde{V}$ is a viscosity solution of the following fully nonlinear parabolic PDE:
 \begin{equation}\label{myq2}
\left\{
\begin{array}
[c]{l}%
\partial_{t}v+ \underset{u\in U}{\inf}[G(H(x,v,\partial_{x}%
v,\partial_{xx}^{2}v,u))+\langle \partial_{x}%
v,b(x,u)\rangle+f(x,v,\partial_{x}%
v\sigma(x,u),u)]=0,\\
v(T,x)=\tilde{V}(x).
\end{array}
\right.
\end{equation}

Then it follows from that the uniqueness of viscosity solution to parabolic PDE \eqref{myq2},  Theorem \ref{my8} and Lemma \ref{myy1} that for each $t\geq 0$, $$\tilde{V}(x)=\underset{u\in\mathcal{U}[0,t]}{\inf}%
\mathbb{G}_{0,t}^{0,x,u}[\tilde{V}(X_{t}^{x,u})].$$

By the proof of Lemma \ref{le-dpp-3} (see inequality \eqref{myq111}), we can find some constant $l$ independent of $t$ so that for each $u\in\mathcal{U}[0,t]$,
\[
|\mathbb{G}_{0,t}^{0,x,u}[\tilde{V}(X_{t}^{x,u})]-\mathbb{G}_{0,t}^{0,x,u}[V(X_{t}^{x,u})]|
\leq l(1+|x|^2)\exp(-\mu t).
\]
Then for each $t\geq 0$, we have
\[
|\tilde{V}(x)-{V}(x)|\leq l(1+|x|^2)\exp(-\mu t).
\]
Letting $t\rightarrow\infty$, we get that $V(x)=\tilde{V}(x)$ for each $x\in\mathbb{R}^n$. The proof is complete.
\end{proof}
\begin{remark}{\upshape
Remark that   \cite{RZ} recently established the well-posedness  of viscosity solutions  of fully nonlinear elliptic path-dependent PDEs under some uniformly continuous conditions, which provides a powerful approach for studying  non-Markovian stochastic  control problem with infinite horizon.
}
\end{remark}

Finally, we shall give the following stochastic verification theorem under the case that the value function $V$ is smooth enough.
\begin{theorem}
Assume \emph{(B1)-(B5)} hold. Suppose that $V$ is a $C^2$-function such that
$\partial_{x^{\mu}x^{\nu}}^{2}V$ is a function of polynomial growth
for any $\mu,\nu=1,\cdots,n$. Then an admissible control $u^*\in\mathcal{U}[0,\infty) $ is optimal if
\[
G(H(\Theta_s^{x,u^*},u^*_s))+\langle \partial_{x}%
V(X_s^{x,u^*}),b(X_s^{x,u^*},u^*_s)\rangle+f(\Theta_s^{x,u^*},u^*_s)]=0,\ \ a.e.\ s\geq 0, \ q.s.,
\]
where $\Theta_s^{x,u}:=(X_s^{x,u},V(X_s^{x,u}),\partial_xV(X_{s}^{x,u})\sigma(X_{s}^{x,u},u_s))$ for each $u\in\mathcal{U}[0,\infty)$.
\end{theorem}
\begin{proof}
Without loss of generality, assume $d=1$.
Then recalling the definition of $H$ and
applying $G$-It\^{o} formula \ref{Thm6.5} to $V(X_t^{x,u^*})$ yields that
\begin{align*}
V(X_t^{x,u^*})=&V(X_T^{x,u^*})-\int^T_t \partial_xV(X_{s}^{x,u^*}) b(X_{s}^{x,u^*},u_s^*) ds -\int^T_t\partial_xV(X_{s}^{x,u^*})\sigma(X_{s}^{x,u^*},u_s^*) dB_s\\
&-\frac{1}{2}\int^T_t H(\Theta_s^{u^*},u_s^*)d\langle B\rangle_s+\int^T_t g(\Theta_s^{u^*},u_s^*)d\langle B\rangle_s\\
=&V(X_T^{x,u^*})+\int^T_t f(\Theta_s^{x,u^*},u^*_s) ds+\int^T_t g(\Theta_s^{u^*},u_s^*)d\langle B\rangle_s \\
&-\int^T_t\partial_xV(X_{s}^{x,u^*})\sigma(X_{s}^{x,u^*},u_s^*) dB_s-(K_T^{u^*}-K_t^{u^*}),
\end{align*}
where $K_t^{u^*}=\frac{1}{2}\int^t_0 H(\Theta_s^{u^*},u^*_s)d\langle B\rangle_s-\int^t_0 G(H(\Theta_s^{u^*},u^*_s))ds$ is a decreasing $G$-martingale.
Thus by the uniqueness of $G$-BSDE with infinite horizon, we conclude that
$V(X_t^{x,u^*})=Y^{x,u^*}_t  q.s.
$
In particular  \[
V(x)=\underset{u\in\mathcal{U}[0,\infty)}{\inf}Y_{0}^{x,u}=Y^{x,u^*}_0,
\]
which completes the proof.
\end{proof}

\begin{example}{\upshape
Consider the following simple infinite horizon discounted stochastic linear model:
\[
J(x,u)=\mathbb{\hat{E}}[\int^{\infty}_0 \exp(-\lambda t)(X^{x,u}_t-u_t)dt],
\]
where $\lambda>0$,  $d=1$, $b(x,u)=- x+ u, h(x,u)=0$, $\sigma(x,u)=x+u$ and $U=[0,1]$. Thus taking $f(x,y,u)=-\lambda y+x-u$ and $g=0$ in equation \eqref{App1m} as in the introduction, we have
\[
J(x,u)=Y^{x,u}_0, \ \forall u\in\mathcal{U}[0,\infty).
\]
Note that $V(x)=\frac{1}{\lambda+1}(x-1)$ is the classical solution to the following equation
\[
\underset{u\in U}{\inf}[G((x+u)^2\partial^2_{xx}V(x))- x\partial_{x}%
V(x)+ u \partial_{x}%
V(x)-\lambda V(x)+ x-u]=0.
\]
Since
\[
G((x+u)^2\partial^2_{xx}V(x))- x\partial_{x}V(x)+  \partial_{x}%
V(x)-\lambda V(x)+ x-1=0,
\]
we deduce that $u^*_s=1, s\geq 0$ is an optimal control.
}
\end{example}
\begin{remark}
{\upshape
Note that  \cite{Fleming W.H} also studied the existence of optimal Markov control policy, i.e., $u^*_s=u^*(X^*_s)$ in a weak framework (see also \cite{J.Yong}).
However, in general we cannot get a optimal Markov control policy since our formulation is a
``strong'' framework.
}
\end{remark}

\textbf{Acknowledgement}: The authors would like to thank  Prof. Shige Peng for   his helpful discussions and
suggestions. The authors also thank the editor and the anonymous referee for their careful reading,  helpful
suggestions.

\appendix
\renewcommand\thesection{Appendix}
\section{ }
\renewcommand\thesection{A}
\subsection{The proof of Theorem \ref{HM3}}
\begin{proof}We shall only prove the existence, since the uniqueness can be proved in a similar way as in \cite{HW}. Without loss of generality, we assume that $g_{ij}=0$.
For convenience, we omit superscripts $t$, $\xi$ and $u$.

Denote by $(Y^n , Z^n, K^n)\in \mathfrak{S}_{G}^{2}(0,n)$ the unique solution of the
following $G$-BSDE in the interval $[t,n]$:
\begin{align*}
Y^n_s=\int^n_sf(X_r,Y^n_r,Z^n_r,u_r)dr-\int^n_sZ^n_rdB_r-(K_n^n-K^n_s).
\end{align*}
Setting ${f}_{s}=f(X_s,Y_{s}^{n},Z_{s}%
^{n},u_s)-f(X_s,0,0,u_s)$, we have
\[
Y^n_s=\int_{s}^{n}(f(X_r,0,0,u_r)+{f}_{r})dr-\int_{s}^{n}{Z}^n_{r}dB_{r}%
-(K^n_{n}-K_{s}^n),
\]
 Then by Lemma 3.5 in \cite{HW},  for each $\varepsilon>0$, we can get that
\[
{f}_{s}=a_{s}^{n,\varepsilon}Y^n_{s}+b_{s}^{n,\varepsilon}Z^n
_{s}-m_{s}^{n,\varepsilon},
\]
where $a_{s}^{n,\varepsilon},b_{s}^{n,\varepsilon},m_{s}^{n,\varepsilon}$ are in $M^2_G(0,T)$ for each $T>t$ and $a_{s}^{n,\varepsilon}\leq -\mu$,
 $|b_{s}^{n,\varepsilon}|\leq \alpha_2$ and $|m_{s}^{n,\varepsilon}|\leq 2(L+\alpha_2)\varepsilon$. Thus applying Lemma \ref{the5.2}, we derive that in the
extended $\tilde{G}$-expectation space,
\begin{equation*}
Y^n_{s}=\mathbb{\hat{E}}_{s}^{\tilde{G}}[\int_{s}^{n}(f(X_r,0,0,u_r)+m_r^{n,\varepsilon})\exp(\int^r_s a_{l}^{n,\varepsilon}dl){\Gamma}_r^{s,n,\varepsilon}dr] \ \ q.s.
\end{equation*}
where \[
{\Gamma}^{s,n,\varepsilon}_{r}=1+\int_{s}%
^{r}b^{n,\varepsilon}_{l}{\Gamma}^{s,n,\varepsilon}_{l}d\tilde{B}_{l}.
\]
Thus  we deduce that
\begin{align*}
|Y^n_{s}|
\leq \mathbb{\hat{E}}^{\tilde{G}}_s[\int^n_s e^{-\mu (r-s)}{\Gamma}^{s,n,\varepsilon}_r |f(X_r,0,0,u_r)| dr]+\frac{2(L+\alpha_2)}{\mu}\varepsilon.
\end{align*}
On the other hand, it follows from Lemma \ref{HW2} that for each $t\leq s\leq r$,
\[
 \mathbb{\hat{E}}^{\tilde{G}}_s[|f(X_r,0,0,u_r)|\Gamma^{s,n,\varepsilon}_r]\leq C_{f}\mathbb{\hat{E}}^{\tilde{G}}_s[(1+|X_r|^2)\Gamma^{s,n,\varepsilon}_r] \leq C_{f}+C_{f} C_{\eta}(1+|X_s|^2),
\]
where $C_f$ is a constant depending on $f$.
Then letting $\varepsilon\rightarrow0$,
we can obtain that
\begin{align*}
|Y^n_{s}|\leq \frac{C_{f}}{\mu}(1+C_{\eta})(1+|X_s|^2), \ \ q.s.
\end{align*}

Now we define $Y^n$, $Z^n$ and $K^n$ on the
whole time axis by setting
\[
Y^n_s=Z^n_s=0,\ K^n_s=K^n_n, \ \ \forall s>n.
\]
 Therefore using the same strategy as in \cite{HW} implies that for each $t\leq s\leq n\leq m$,
\begin{align*}
|Y^n_s-Y^m_s|\leq \frac{C_{f}}{\mu}(1+C_{\eta})(1+|X_s|^2)\exp(\mu s)(\exp(-\mu n)-\exp(-\mu m)), \ \ q.s.
\end{align*}
Thus, we get for each $t<T\leq n\leq m$,
\[
\lim\limits_{m,n\rightarrow\infty}\mathbb{\hat{E}}[\sup\limits_{s\in[t,T]}|Y^n_s-Y^m_s|^2]=0.
\]
In sprit of Proposition 3.8 in \cite{HJPS},  we  conclude that
 \[
 \lim\limits_{m,n\rightarrow\infty}\|Z^n-Z^m\|_{M_G^2(t,T)}=0.
 \]
Consequently, there exist  two processes $(Y,Z)\in\mathcal{S}_G^2(0,\infty)\times M_G^2(0,\infty)$ such that
\[
\lim\limits_{n\rightarrow\infty}\mathbb{\hat{E}}[\sup\limits_{s\in[t,T]}|Y^n_s-Y_s|^2+\int^T_t|Z^n_s-Z_s|^2dt]=0.
\]
It is obvious that $|Y_{s}|\leq \frac{C_{f}}{\mu}(1+C_{\eta})(1+|X_s|^2)$ for each $s\geq t$ q.s. Denote
\[
K_s:=Y_s-Y_t+\int^s_tf(X_r,Y_r,Z_r,u_r)dr-\int^s_tZ_rdB_r.
\]
One can easily check that $(Y,Z,K)$ is the solution to equation \eqref{App1m}.
\end{proof}
\subsection{$G$-Stochastic optimal control problem in finite horizon}
This section is devoted to extending the results in \cite{HJ1} to the case that the terminal condition is  a  continuous function of quadratic growth. The main idea is based on \cite{HW}, where there is no control.
For some fixed $T>0$ and for each $(t,x,u)\in[0,T]\times\mathbb{R}^n\times\mathcal{U}[t,T]$, consider the following type of $G$-BSDEs on finite interval $[t,T]$:
\begin{align}\label{my2}
&Y_{s}^{t,T,x,u}   =\phi(X^{t,x,u}_T)+\int_{s}^{T}f(X_{r}^{t,x,u}%
,Y_{r}^{t,T,x,u},Z_{r}^{t,T,x,u},u_r)dr-\int_{s}^{T}Z_{r}^{t,T,x,u}dB_{r}\nonumber\\
& \ \ \ +\int_{s}^{T}g_{ij}(X_{r}^{t,x,u}%
,Y_{r}^{t,T,x,u},Z_{r}^{t,T,x,u},u_r)d\langle B^i,B^j\rangle_{r} -(K_{T}^{t,T,x,u}-K_{s}^{t,T,x,u}),
\end{align}
where $\phi$ is a continuous function such that $|\phi(x)| \leq M(1+|x|^2)$ for some constant $M$. By Theorem \ref{the4.1},  the equation
\eqref{my2} has a unique solution $(Y^{t,T,x,u},Z^{t,T,x,u},K^{t,T,x,u})\in \mathfrak{S}_{G}^{2}(t,T)$.
For convenience, we set $(Y^{T,x,u},Z^{T,x,u},K^{T,x,u})=(Y^{0,T,x,u},Z^{0,T,x,u},\\ K^{0,T,x,u})$.
Then we denote \[\bar{V}(t,x)=\underset{u\in\mathcal{U}[t,T]}{ess\inf}Y^{t,T,x,u}_t.\]

Note that there exists a sequence  Lipschitz functions $\{\phi^m\}_{m=1}^{\infty}$  such that
\[
|\phi(x)-\phi^m(x)|\leq \frac{1}{m}\mathbf{1}_{\{|x|\leq m\}}+2M(1+|x|^2)\mathbf{1}_{\{|x|>m\}}\leq  \frac{1}{m}+\frac{2M(1+|x|^3)}{m}.
\]
Let $(Y^{t,T,m,x,u},Z^{t,T,m,x,u},K^{t,T,m,x,u})$ be the unique $\mathfrak{S}_{G}^{2}(t,T)$-solution of $G$-BSDEs \eqref{my2} with terminal condition
$Y^{t,T,m,x,u}_T=\phi^m(X^{t,x,u}_T)$ and denote $$\bar{V}^m(t,x)=\underset{u\in\mathcal{U}[t,T]}{ess\inf}Y^{t,T,m,x,u}_t.$$

\begin{lemma}\label{my3}
Under assumptions \emph{(B1)-(B2)}, $\bar{V}^m(t,x)$ is a deterministic continuous function. Moreover,
 $\bar{V}^m(t,x)$ is the unique viscosity solution of the following fully nonlinear PDE with terminal condition $\bar{V}^m(T,x)=\phi^m(x)$:
\begin{equation}
\left\{
\begin{array}
[c]{l}%
\partial_{t}v+ \underset{u\in U}{\inf}[G(H(x,v,\partial_{x}%
v,\partial_{xx}^{2}v,u))+\langle \partial_{x}%
v,b(x,u)\rangle+f(x,v,\partial_{x}%
v\sigma(x,u),u)]=0,\\
v(T,x)=\phi(x).
\end{array}
\right.  \label{myfeynman}%
\end{equation}
\end{lemma}
\begin{proof}
The proof is immediate from \cite{HJ1} and \cite{HJPS1}.
\end{proof}

Now we shall state the main result of this appendix.
\begin{theorem} \label{my8}
 Assume \emph{(B1)-(B2)} hold. Then $\bar{V}(t,x)$ is the unique viscosity solution of the  fully nonlinear PDE
\eqref{myfeynman} with terminal condition $\bar{V}(T,x)=\phi(x)$.
\end{theorem}

In order to prove this theorem, we need the following lemmas.

\begin{lemma}\label{my4}
For each   function $\varphi\in C(\mathbb{R}^n)$ with quadratic growth, $\mathbb{\hat{E}}[\varphi(X^{x,u}_t)]$ is a continuous function of $(t,x,u)\in[0,T]\times\mathbb{R}^n\times\mathcal{U}[0,T]$.
\end{lemma}
\begin{proof}
The proof is similar to Lemma A.5 in \cite{HW}. For convenience, we shall give the sketch of the proof.
Assume $|\varphi(x)|\leq C_{\varphi}(1+|x|^2)$, where $C_{\varphi}$ is generic constant depending on $\varphi$ and may vary from line to line.
For each given $N>0$ and $T>0$, for any $t,t^{\prime}< T$, $x,x^{\prime}\in\mathbb{R}^n$ and $u,u^{\prime}\in\mathcal{U}[0,T]$, we have
\begin{align*}
&|\mathbb{\hat{E}}[\varphi(X^{x,u}_t)]-\mathbb{\hat{E}}[\varphi(X^{x^{\prime},u^{\prime}}_{t^{\prime}})]|\\&
\leq \mathbb{\hat{E}}[|\varphi(X^{x,u}_t)-\varphi(X^{x^{\prime},u^{\prime}}_{t^{\prime}})|\mathbf{1}_{\{|X^{x,u}_t|\leq N\}\cap \{|X^{x^{\prime},u^{\prime}}_{t^{\prime}}|\leq N\}}]+\frac{ C_{\varphi}}{N}\mathbb{\hat{E}}[1+|X^{x^{\prime},u^{\prime}}_{t^{\prime}}|^3+|X^{x,u}_t|^3].
\end{align*}
Note that for each given $\epsilon>0$, there is $\rho>0$ such that
\[
|\varphi(z)-\varphi(z^{\prime})|\leq \frac{\epsilon}{2} \ \text{whenever $|z-z^{\prime}|<\rho$ and $|z|,|z^{\prime}|\leq N$}.
\]
From Lemma  \ref{myq.1} and H\"{o}lder's inequality, there is a constant $\delta>0$ such that
\[
\mathbb{\hat{E}}[|\varphi(X^{x,u}_t)-\varphi(X^{x^{\prime},u^{\prime}}_{t^{\prime}})|\mathbf{1}_{\{|X^{x,u}_t-X^{x^{\prime},u^{\prime}}_{t^{\prime}}|\geq \rho\}}]<\frac{\epsilon}{2}
\]
whenever $|x-x^{\prime}|\leq \delta$, $|t-{t^{\prime}}|\leq \delta$ and $\|u-u^{\prime}\|_{M^2_G(0,T)}\leq \delta$.
Consequently,
\begin{align*}
&|\mathbb{\hat{E}}[\varphi(X^{x,u}_t)]-\mathbb{\hat{E}}[\varphi(X^{x^{\prime},u^{\prime}}_{t^{\prime}})]|\\&
\leq \mathbb{\hat{E}}[|\varphi(X^{x,u}_t)-\varphi(X^{x^{\prime},u^{\prime}}_{t^{\prime}})|\mathbf{1}_{\{|X^{x,u}_t-X^{x^{\prime},u^{\prime}}_{t^{\prime}}|<\rho\}
\cap\{|X^{x,u}_t|\leq N\}\cap \{|X^{x^{\prime},u^{\prime}}_{t^{\prime}}|\leq N\}}]\\ &\  +\mathbb{\hat{E}}[|\varphi(X^{x,u}_t)-\varphi(X^{x^{\prime},u^{\prime}}_{t^{\prime}})|\mathbf{1}_{\{|X^{x,u}_t-X^{x^{\prime},u^{\prime}}_{t^{\prime}}|\geq \rho\}}]
+\frac{C_{\varphi}}{N}\mathbb{\hat{E}}[1+|X^{x^{\prime},u^{\prime}}_{t^{\prime}}|^3+|X^{x,u}_t|^3]\\&
\leq \epsilon+\frac{C_{\varphi}}{N}\mathbb{\hat{E}}[1+|X^{x^{\prime},u^{\prime}}_{t^{\prime}}|^3+|X^{x,u}_t|^3]
\end{align*}
whenever $|x-x^{\prime}|\leq \delta$, $|t-{t^{\prime}}|\leq \delta$ and $\|u-u^{\prime}\|_{M^2_G(0,T)}\leq \delta$. Thus we get
\[
\limsup\limits_{(t^{\prime},x^{\prime},u^{\prime})\rightarrow(t,x,u)}|\mathbb{\hat{E}}[\varphi(X^{x,u}_t)]-\mathbb{\hat{E}}[\varphi(X^{x^{\prime},u^{\prime}}_{t^{\prime}})]| \leq \epsilon+\frac{C_{\varphi}}{N}\mathbb{\hat{E}}[1+|X^{x,u}_t|^3].
\]
Then we  obtain the desired result by letting $\epsilon \downarrow 0$ and then sending $N\rightarrow\infty$.
\end{proof}

\begin{lemma}
\label{the-dpp-1} Assume \emph{(B1)-(B2)} hold. Then the value function $\bar{V}(t,x)$ exists and
\[
\bar{V}(t,x)=\inf_{u\in\mathcal{U}^{t}[t,T]}Y_{t}^{t,T,x,u}.
\]
\end{lemma}
\begin{proof}
 Assume $u\in\mathcal{U}[t,T]$ and $u^k\in\mathcal{U}[t,T]$ for $k\geq 1$.
 Then it follows from Lemma \ref{my4} and Theorem \ref{pro3.5} (see also Theorem 7 in \cite{HJ1}) that $Y^{t,T,x,u^k}_t$ converges to $Y^{t,T,x,u}_t$ in $L_G^2(\Omega_t)$
 whenever $u^k$ converges to $u$ in $M^2_G(0,T)$ as $k\rightarrow\infty$. Then one can complete the proof by the same way as in Theorem 17 of \cite{HJ1}.
\end{proof}
\begin{lemma}\label{my7}
Assume \emph{(B1)-(B2)} hold. Then the following properties hold:
\begin{description}
\item[(i)] There exists a constant $C$ depending on $M$, $T$, $G$, $L$, $\alpha_1$ and $\alpha_2$ such that \[\|Y^{t,T,m,x,u}\|_{S_G^2(t,T)}+\|Z^{t,T,m,x,u}\|_{M_G^2(t,T)}\leq C(1+|x|^2), \ \forall x\in\mathbb{R}^n,m\geq 1, u\in\mathcal{U}[t,T];\]
\item[(ii)] $\lim\limits_{m\rightarrow\infty}\mathbb{\hat{E}}[\sup\limits_{s\in[t,T]}|Y^{t,T,m,x,u}_s-Y^{t,T,x,u}_s|^2]=0$;
\item[(iii)] $\bar{V}(t,x)$ is a  continuous function of quadratic growth;
\item[(iv)] $\lim\limits_{m\rightarrow \infty}\bar{V}^m(t_m,x_m)=\bar{V}(t,x)$ for each given $(t_m,x_m)\in[0,T]\times\mathbb{R}^n$ with $(t_m,x_m)\rightarrow (t,x)$.
\end{description}
\end{lemma}
\begin{proof}
Note that $\phi^m$ and $f(x,0,0,u),g_{ij}(x,0,0,u)$ are functions of quadratic growth in $x$ with uniformly bounded coefficients. Applying Proposition 3.5 and Corollary 5.2 in \cite{HJPS}, we obtain (i).
By Theorem \ref{pro3.5}  and Theorem 3.3 in \cite{Song11}, we can find a constant $\tilde{C}$  depending on $M$, $T$, $G$, $L$, $\alpha_1$ and $\alpha_2$ (may vary from line to line), such that, for any $(t,x,u)\in[0,T]\times\mathbb{R}^n\times\mathcal{U}[0,T]$,
\begin{align}
&\lim\limits_{m\rightarrow\infty}\mathbb{\hat{E}}[\sup\limits_{s\in[t,T]}|Y^{t,T,m,x,u}_s-Y^{t,T,x,u}_s|^2]
\nonumber\\&  \leq \lim\limits_{m\rightarrow\infty}\tilde{C} ((\mathbb{\hat{E}}[|\phi(X^{t,x,u}_T)-\phi^m(X^{t,x,u}_T)|^{3}])^{\frac{2}{3}}+\mathbb{\hat{E}}[|\phi(X^{t,x,u}_T)-\phi^m(X^{t,x,u}_T)|^{3}])\nonumber\\
&\leq \lim\limits_{m\rightarrow\infty}\tilde{C}( \frac{1}{m^2}+\frac{\mathbb{\hat{E}}[|X^{t,x,u}_T|^{9}]+(\mathbb{\hat{E}}[|X^{t,x,u}_T|^{9}])^{\frac{2}{3}}}{m^2})=0.\label{my5}
\end{align}
Moreover, from Lemma  \ref{myq.1} (ii), we obtain that
\begin{align*}
&\lim\limits_{m\rightarrow\infty}\sup_{u\in\mathcal{U}^{t}[t,T]}|Y^{t,T,m,x,u}_t-Y^{t,T,x,u}_t|^2\\
&\leq \lim\limits_{m\rightarrow\infty}\tilde{C}( \frac{1}{m^2}+\sup_{u\in\mathcal{U}^{t}[t,T]}\frac{\mathbb{\hat{E}}[|X^{t,x,u}_T|^{9}]+(\mathbb{\hat{E}}[|X^{t,x,u}_T|^{9}])^{\frac{2}{3}}}{m^2})=0.
\end{align*}
In particular, $\lim\limits_{m\rightarrow \infty}\bar{V}^m(t,x)=\bar{V}(t,x)$ by Lemma \ref{the-dpp-1}.

Now we prove $\lim\limits_{m\rightarrow \infty}\bar{V}(t_m,x_m)=\bar{V}(t,x)$ for each given $(t_m,x_m)\in[0,T]\times\mathbb{R}^n$ with $(t_m,x_m)\rightarrow (t,x)$.
Without loss of generality, we assume $t_m\leq t$ and $g_{ij}=0$. By a similar analysis as in (ii) and Lemma \ref{my4},  we can obtain
\begin{align}\label{my6}
\lim\limits_{m\rightarrow\infty}\sup_{u\in\mathcal{U}[0,T]}&\mathbb{\hat{E}}[\sup\limits_{s\in[t,T]}|Y^{t_m,T,x_m,u}_s-Y^{t,T,x,u}_s|^2
\nonumber\\ & \ \ \ \ \ \ \ \ +\int^T_t|Z^{t_m,T,x_m,u}_s-Z^{t,T,x,u}_s|^2ds]=0. \end{align}

It is easy to check that $\bar{V}(t,x)=\underset{u\in\mathcal{U}[0,T]}{ess\inf}Y^{t,T,x,u}_t$. Since $\bar{V}(t,x)$ is a
deterministic function, taking expectation on both sides of  equation \eqref{my2} yields that
\[
\bar{V}(t,x)=\mathbb{\hat{E}}[\underset{u\in\mathcal{U}[0,T]}{ess\inf}Y^{t,T,x,u}_t]\leq \inf_{u\in\mathcal{U}[0,T]}\mathbb{\hat{E}}[\phi(X^{t,x,u}_T)+\int_{t}^{T}f(X_{r}^{t,x,u}%
,Y_{r}^{t,T,x,u},Z_{r}^{t,T,x,u},u_r)dr].
\]
From Lemma \ref{the-dpp-1}, we derive that
\[
\bar{V}(t,x)=\underset{u\in\mathcal{U}^t[t,T]}{\inf}\mathbb{\hat{E}}[Y^{t,T,x,u}_t]\geq \inf_{u\in\mathcal{U}[0,T]}\mathbb{\hat{E}}[\phi(X^{t,x,u}_T)+\int_{t}^{T}f(X_{r}^{t,x,u}%
,Y_{r}^{t,T,x,u},Z_{r}^{t,T,x,u},u_r)dr],
\]
which implies that
\[
\bar{V}(t,x)=\underset{u\in\mathcal{U}[0,T]}{\inf}\mathbb{\hat{E}}[Y^{t,T,x,u}_t]= \inf_{u\in\mathcal{U}[0,T]}\mathbb{\hat{E}}[\phi(X^{t,x,u}_T)+\int_{t}^{T}f(X_{r}^{t,x,u}%
,Y_{r}^{t,T,x,u},Z_{r}^{t,T,x,u},u_r)dr].
\]
Consequently,
\begin{align*}
&|\bar{V}(t,x)-\bar{V}(t_m,x_m)|\\&
\leq \sup_{u\in\mathcal{U}[0,T]}\mathbb{\hat{E}}[|\phi(X^{t,x,u}_T)-\phi(X^{t_m,x_m,u}_T)|+\int_{{t_m}}^{{t}}|f(X_{r}^{t_m,x_m,u}%
,Y_{r}^{t_m,T,x_m,u},Z_{r}^{t_m,T,x_m,u},u_r)|\,dr\\
&\ \ \  +\int_{{t}}^{T}|f(X_{r}^{t,x,u}%
,Y_{r}^{t,T,x,u},Z_{r}^{t,T,x,u},u_r)-f(X_{r}^{t_m,x_m,u}%
,Y_{r}^{t_m,T,x_m,u},Z_{r}^{t_m,T,x_m,u},u_r)|\,dr]\\
& \leq  \sup_{u\in\mathcal{U}[0,T]}\mathbb{\hat{E}}[(t-t_m)^{\frac{1}{2}}(\int_{{t_m}}^{{t}}3(|f(X_r^{t_m,x_m,u},0,0,u_r)|^2+|LY_{r}^{t_m,T,x_m,u}|^2+|\alpha_2Z_{r}^{t_m,T,x_m,u}|^2)\,dr)^{\frac{1}{2}}\\
&\ \ \ +\int_{{t}}^{T}(L(1+|X_{r}^{t,x,u}|+|X_{r}^{t_m,x_m,u}|)|X_{r}^{t,x,u}-X_{r}^{t_m,x_m,u}|+L
|Y_{r}^{t,T,x,u}-Y_{r}^{t_m,T,x_m,u}| )\,dr\\& \ \ \ +\alpha_2\int_{{t}}^{T}|Z_{r}^{t,T,x,u}-Z_{r}^{t_m,T,x_m,u}|\,dr+
|\phi(X^{t,x,u}_T)-\phi(X^{t_m,x_m,u}_T)|].
\end{align*}
By Lemma \ref{my4}, (i) and equation \eqref{my6}, we derive that
\[
\lim\limits_{m\rightarrow\infty}|\bar{V}(t,x)-\bar{V}(t_m,x_m)|=0,
\]
 and the property (iii) holds.

 From (iii), we get that
\begin{align*}&\lim\limits_{m\rightarrow\infty}|\bar{V}^m(t_m,x_m)-\bar{V}(t,x)|\\ \ \  \ \ \ & \leq \lim\limits_{m\rightarrow\infty}|\bar{V}^m(t_m,x_m)-\bar{V}(t_m,x_m)|+ \lim\limits_{m\rightarrow\infty}|\bar{V}(t_m,x_m)-\bar{V}(t,x)|\\ \ \ \ \ \ &=\lim\limits_{m\rightarrow\infty}|\bar{V}^m(t_m,x_m)-\bar{V}(t_m,x_m)|.
\end{align*}
By Lemma \ref{myq.1} (ii) and equation \eqref{my5}, we obtain
\begin{align*}&
\lim\limits_{m\rightarrow\infty}|\bar{V}^m(t_m,x_m)-\bar{V}(t,x)|\\ \ \ \ \ \  &\leq\lim\limits_{m\rightarrow\infty}\tilde{C}( \frac{1}{m}+\sup_{u\in\mathcal{U}[0,T]}\frac{\mathbb{\hat{E}}[|X^{t_m,x_m,u}_T|^{9}]^{\frac{1}{2}}+\mathbb{\hat{E}}[|X^{t_m,x_m,u}_T|^{9}]^{\frac{1}{3}}}{m})=0.
\end{align*}
 The proof is complete.
\end{proof}

\begin{lemma}\label{myy1}
Assume \emph{(B1)-(B2)} hold. Then for any $t\leq
s\leq T$, $x\in\mathbb{R}^{n}$, we have
\begin{equation*}%
\begin{array}
[c]{rl}%
\bar{V}(t,x)= & \underset{u\in\mathcal{U}[t,s]}{\text{ess}\inf}%
\mathbb{G}_{t,s}^{t,x,u}[\bar{V}(s,X_{s}^{t,x,u})]
=  \underset{u\in\mathcal{U}^{t}[t,s]}{\inf}\mathbb{G}_{t,s}%
^{t,x,u}[\bar{V}(s,X_{s}^{t,x,u})].
\end{array}
\end{equation*}
\end{lemma}
\begin{proof}
From Lemma \ref{the-dpp-1}, it suffices to show that
\[
\bar{V}(t,x)= \underset{u\in\mathcal{U}^{t}[t,s]}{\inf}\mathbb{G}_{t,s}%
^{t,x,u}[\bar{V}(s,X_{s}^{t,x,u})].
\]

By Theorem 21 in \cite{HJ1}, we have
\begin{equation*}%
\begin{array}
[c]{rl}%
\bar{V}^m(t,x)= \underset{u\in\mathcal{U}^{t}[t,s]}{\inf}\mathbb{G}_{t,s}%
^{t,x,u}[\bar{V}^m(s,X_{s}^{t,x,u})].
\end{array}
\end{equation*}
From the proof of Lemma \ref{my7}, we conclude that for each $l>0$,
\[
\lim\limits_{m\rightarrow\infty}\sup\limits_{|x|\leq l}|\bar{V}^m(t,x)-\bar{V}(t,x)|=0.
\]
Thus applying Theorem \ref{pro3.5}  and Lemma \ref{myq.1} (ii), we can find some constant $\tilde{C}$ independent of $m$ and $l$ (may vary from line to line)  so that
\begin{align*}
&\lim\limits_{m\rightarrow\infty}\underset{u\in\mathcal{U}^{t}[t,s]}{\sup}|\mathbb{G}_{t,s}%
^{t,x,u}[\bar{V}^m(s,X_{s}^{t,x,u})]-\mathbb{G}_{t,s}%
^{t,x,u}[\bar{V}(s,X_{s}^{t,x,u})]|
\\ &\leq \lim\limits_{m\rightarrow\infty}\tilde{C}\underset{u\in\mathcal{U}^{t}[t,s]}{\sup}\mathbb{\hat{E}}[|\bar{V}^m(s,X_{s}^{t,x,u})-\bar{V}(s,X_{s}^{t,x,u})|^2]^{\frac{1}{2}}
\\ &\leq \tilde{C}\lim\limits_{m\rightarrow\infty}[\sup\limits_{|x|\leq l}|\bar{V}^m(t,x)-\bar{V}(t,x)|+\underset{u\in\mathcal{U}^{t}[t,s]}{\sup}\frac{\mathbb{\hat{E}}[1+|X_{s}^{t,x,u}|^6]^{\frac{1}{2}}}{l}]
\\ &\leq \frac{\tilde{C}}{l}.
\end{align*}
Sending $l\rightarrow\infty$ yields that \[
\lim\limits_{m\rightarrow\infty}\underset{u\in\mathcal{U}^{t}[t,s]}{\inf}\mathbb{G}_{t,s}%
^{t,x,u}[\bar{V}^m(s,X_{s}^{t,x,u})]=\underset{u\in\mathcal{U}^{t}[t,s]}{\inf}\mathbb{G}_{t,s}%
^{t,x,u}[\bar{V}(s,X_{s}^{t,x,u})],
\]
which completes the proof.
\end{proof}

Now we are in a position to give the proof of Theorem \ref{my8}.

\begin{proof}[Proof of Theorem \ref{my8}]
By Lemmas \ref{my3}, \ref{my7} and Proposition 4.3 in \cite{CMI}, it is easy to verify that
$\bar{V}$ is a viscosity solution of the  fully nonlinear PDE
\eqref{myfeynman} with terminal condition $\bar{V}(T,x)=\phi(x)$.
The uniqueness can be found in Theorem 6.1 of \cite{BJ} and
the proof is complete.
\end{proof}

\end{document}